\documentclass{amsart}

\usepackage{mathtools}
\usepackage{enumitem}
\usepackage[colorlinks=true,linkcolor=black,anchorcolor=black,citecolor=black,filecolor=black,menucolor=black,runcolor=black,urlcolor=black]{hyperref}
\usepackage[utf8]{inputenc}
\usepackage[margin=3cm]{geometry}
\usepackage{array}
\usepackage{amssymb}
\usepackage[table]{xcolor}
\usepackage{todonotes}
\usepackage{graphicx}
\usepackage{wrapfig}
\usepackage{caption}
\graphicspath{{Images/}}
\newtheorem{theorem}{Theorem}[section]
\newtheorem{lemma}[theorem]{Lemma}

\theoremstyle{definition}
\newtheorem{assumption}[theorem]{Assumption}
\newtheorem{definition}[theorem]{Definition}
\theoremstyle{remark}
\newtheorem{remark}[theorem]{Remark}

\numberwithin{equation}{section}

\newcommand{\R}{\mathbb{R}}
\newcommand{\norm}[1]{\left\lVert#1\right\rVert}

\begin{document}

\title{An Efficient Model for Scaffold-Mediated Bone Regeneration}
	
	\author{Patrick Dondl}
	\address[Patrick Dondl]{Abteilung f\"ur Angewandte Mathematik,
		Albert-Ludwigs-Universit\"at Freiburg, Hermann-Herder-Strasse 10, 79104 Freiburg i.\ Br.}
	\email{patrick.dondl@mathematik.uni-freiburg.de}
	\urladdr{https://aam.uni-freiburg.de/agdo/}
	
	\author{Patrina S.\ P.\ Poh}
	\address[Patrina S.\ P.\ Poh]{Julius Wolff Institute for Biomechanics and Musculoskeletal Regeneration, Charité – Univeristätsmedizin Berlin, Berlin, Germany}
	\email{patrina.poh@charite.de}
	\urladdr{https://jwi.charite.de/metas/person/person/address\_detail/poh/}
	
	\author{Marius Zeinhofer}
	\address[Marius Zeinhofer]{Abteilung f\"ur Angewandte Mathematik,
		Albert-Ludwigs-Universit\"at Freiburg, Hermann-Herder-Strasse 10, 79104 Freiburg i.\ Br.}
	\email{marius.zeinhofer@mathematik.uni-freiburg.de}
	
	
	
	\date{\today}
	
	\begin{abstract}
	    We present a three dimensional, time dependent model for bone regeneration in the presence of porous scaffolds to bridge critical size bone defects. Our approach uses homogenized quantities, thus drastically reducing computational cost compared to models resolving the microstructural scale of the scaffold. Using abstract functional relationships instead of concrete effective material properties, our model can incorporate the homogenized material tensors for a large class of  scaffold microstructure designs.
	    We prove an existence and uniqueness theorem for solutions based on a fixed point argument. We include the cases of mixed boundary conditions and multiple, interacting signalling molecules, both being important for application. Furthermore we present numerical simulations showing good agreement with experimental findings.
	\end{abstract}
	\maketitle
	\tableofcontents

\section{Introduction}

In this work, we are concerned with the development and well-posedness of a simple and efficient model for bone regeneration in the presence of a bioresorbable porous scaffold. The essential processes are an interplay between the mechanical and biological environment which we model by a coupled system of PDEs and ODEs. The mechanical environment is represented by a linear elastic equation and the biological environment through reaction-diffusion equations as well as as logistic ODEs, modelling signalling molecules and cells/bone respectively. Material properties are incorporated using homogenized quantities not resolving any scaffold microstructure. This makes the model efficient in computations, thus suitable as a forward equation in optimization algorithms and opening  up the possibility of patient specific scaffold design in the sense of precision medicine. 

We analyze the model mathematically, proving well-posedness. We stress that we allow data that is realistic for applications, i.e., non-smooth domains and mixed Dirichlet-Neumann boundary conditions.

The article is organized as follows. Next, we give an introduction into tissue engineering for the treatment severe bone defects and present our computational model. Then, we discuss its weak formulation in section \ref{sectionMathematicalFormulation} and prove an existence and uniqueness result in section \ref{sec:existence}. Finally, numerical simulations are presented in \ref{sectionNumericalExperiments}. Appendix \ref{sec:AppendixDiffusion} is concerned with regularity results for Dirichlet-Neumann boundary value problems and Appendix \ref{sec:AppendixODE} contains results on Banach space valued ODEs. We include the latter because, even though the results are folklore, we are not aware of any references.

	\subsection{Scaffold Mediated Bone Growth}
	The regeneration and restoration of skeletal functions of critical-sized bone defects ($>$25 mm) are very challenging despite a multitude of treatment options \cite{nauth2018critical}. The main problem is the phenomenon of non-union where the bone defect fails to become bridged after $>$9 months and does not show healing progression for 3 months \cite{calori2017non}. With $1.9\%$, the prevalence of non-union per fracture is relatively low \cite{mills2017risk}, yet the financial burden is high, for example, in the UK, the healthcare cost is estimated to be $\pounds 320$ million annually \cite{stewart2019fracture}. Moreover, the risk of non-union increases drastically with comorbidities such as diabetes as in this case the regenerative capability of bone tissue is compromised \cite{marin2018impact}.   
	
	Critical-sized defects may not heal and require in-depth planning of their treatment. Currently used therapeutic approaches include bone grafting, distraction osteogenesis, and the so-called ``Masquelet'' technique, in which a periosteal membrane is formed to induce bone defect healing \cite{nauth2018critical}. Despite having a general guideline for treatment of critical-sized bone defects, healing outcomes vary highly, dependent on the site and size of the defect and patient-related aspects, e.g., age, lifestyle and comorbid metabolic/systemic disorders \cite{roddy2018treatment}.

    Over the years, research illustrated the potential of using porous, possibly bio-resorbable support structures, so-called scaffolds, as supporting devices to promote bone defect regeneration. Initially, a scaffold is placed in the defect site, acting as a temporary support structure allowing for vascularization while guiding new bone formation. This has recently shown promising results in vivo and in clinical cases, for example \cite{petersen2018biomaterial} showed that the architecture of the scaffold can guide the endochondral healing of bone defects in rats. In this study, collagen-based scaffolds with cylindrical pores aligned along the principle stress axis were used. In \cite{cipitria2012porous, paris2017scaffold}, 3D-printed scaffolds made from a composite of polycaprolactone (PCL, a slowly degrading, bio-resorbable synthetic thermoplastic) and $\beta$-tricalcium phosphate ($\beta$-TCP) were used in an ovine experiment. In the studies \cite{petersen2018biomaterial, cipitria2012porous, paris2017scaffold} no relevant bridging of the bone defect was achieved without the addition of exogenous growth factors or cells. However, \cite{pobloth2018mechanobiologically} illustrated that clinically relevant bone formation for scaffold mediated bone regeneration is possible without exogenous growth factors. In this experiment a 3D-printed titanium scaffold with optimized mechanobiological properties was used and displayed clinically relevant functional bridging of a major bone defect in a large animal model. Concluding, these studies \cite{petersen2018biomaterial, cipitria2012porous, paris2017scaffold, pobloth2018mechanobiologically} indicate the possibility of using a scaffold-mediated bone growth approach for critical-size bone defect healing. Furthermore they indicate that the design and choice of materials are critical questions not yet fully understood.

	There are several objectives to be considered when designing a scaffold, such as (a) the porosity, pore size and shape, influencing cell proliferation and differentiation as well as the vascularization process; (b) the overall stability and elastic properties guaranteeing a proper transfer of loads, as mechanical stimulus is indispensable for bone growth; (c) patient specific information such as reduced bone healing capacities, caused for example by diabetes \cite{marin2018impact}. Therefore, the patient dependent optimal scaffold design is of fundamental importance and with the advent of additive manufacturing technologies the production of personalized scaffolds is -- in theory -- fully feasible.

	However, the design of scaffolds has been dominated by trial-and-error approaches -- modifying an existing scaffold architecture based on experimental outcomes, a very costly workflow unsuitable for patient specific design. Over the years, with the help of evolving computer aided design tools, topology optimization techniques have shown potential to address the optimal design question computationally. 
	
	This strategy has already been applied to design scaffolds meeting elastic optimality conditions with a given porosity or fluid permeability \cite{dias2014optimization, coelho2015bioresorbable, lin2004novel, guest2006optimizing, challis2012computationally, kang2010topology,Wang:2016js, dondl_bone_shape_opt}.
	Yet, a common limitation to these models is that they do not resolve the time dependence of the bone regeneration process, as scaffold mediated bone regeneration crucially depends on the varying elastic moduli over time.

	Highly accurate, fine scale models for bone formation exist (see, e.g., \cite{perez1, Sanz-Herrera.2008, ALIERTA2014328, CHECA2010961}). A central issue in most such micro-scale models is that their use in optimization routines for scaffold design is impeded by too high computational cost. Ideally, a bone regeneration scaffold design should be patient specific, i.e., depend on the individual patient's defect site and its biomechanical loading conditions, geometry, and regenerative ability as influenced by, e.g., comorbitities such as type 2 diabetes mellitus. Such an optimization of course relies on the availability of highly efficient models for bone regeneration that nevertheless take into account mechanics and biological signalling.
	
	Based on a previous, one-dimensional study \cite{poh2019optimization}, we thus propose a model based on homogenized quantities suitable for scaffold optimization in the sense of the first step in the ``Shape Optimization by the Homogenization Method'' \cite{allaire2012shape}. This means that our model does not resolve the micro-structure of the scaffold design, but uses coarse-grained values instead. In a scaffold based on a unit cell design, the scaffold volume fraction (or equivalently, the porosity) changes on a larger length-scale than the unit cell design. We use this fact to simplify our model, working with meso-scale averages of the volume fraction instead of the precise micro-structure. Likewise, the other quantities of the model can be viewed as locally averaged values. However, it should be made clear that using such an approach implies that only the averaged quantities can be tracked over the regeneration process and no prediction on how the micro-structure changes over time can be made. Rather, this is required as an input to provide the correct homogenized material properties. Our central assumption is that one can describe the time-evolution of the homogenized quantities in terms of their averages at the initial time-point. Compared to the aforementioned one-dimensional approach, our model can resolve important issues such as bone mass loss due to stress shielding in orthopaedic implants, see section \ref{sectionNumericalExperiments} for an explicit example.
    
    As our model is designed for computational efficiency we include only key events in the course of the bone healing process. We keep track of the mechanical environment at every point in time and space, depending on the current state of bone formation and scaffold degradation in terms of its molecular weight. Here we focus on additively manufactured scaffolds made out of PCL, a very promising material for this specific application. Of course, extensions to other materials (e.g., non-degrading titanium) are possible. The biological environment is represented via a concentration of endogenous angiogenic and osteoinductive factors (e.g., intrinsic growth factors/cytokines) which we call bio-active or signalling molecules and a concentration of osteoblasts, a type of bone forming cell. The coupling of the mechanical and biological properties is assumed to be driven through the local strain caused by mechanical loading of the scaffold-bone composite, i.e., mechanical loading leads to stimulus for the biological environment which in turn leads to bone growth and hence changes the mechanical properties.

    This results in a coupled system of evolution equations composed of a linear elastic equilibrium equation for every point in time, diffusion equations for the bio-active molecules and ordinary differential equations for the concentration of osteoblasts and the volume fraction of bone. As our main mathematical result we prove that this system admits a unique solution in a certain weak sense, see Theorem \ref{ExistenceThm}. This shows that our model is well-posed, a necessary requirement for a reasonable biological model. The strategy used to prove Theorem \ref{ExistenceThm} is to apply a fixed-point theorem on a map associated with the coupled system of equations, see the beginning of section \ref{sec:existence} for a precise description. The main difficulty we encounter is the notorious low regularity of mixed Dirichlet-Neumann boundary value problems \cite{savare1997regularity, kassmann2004regularity, grisvard2011elliptic} which one is forced to consider when one desires to allow for realistic boundary conditions See Appendix \ref{sec:AppendixDiffusion} where we collect results from the literature that are helpful in our case. 
    
    As our main focus lies on the existence and uniqueness results, we do not use concrete homogenized tensors in the equations, but abstract functional relationships. This has the advantage of proving the result for a wide class of imaginable scaffold architectures at once. The concrete micro-structure can then be taken into account when one performs numerical simulations. In the same spirit we keep the rest of the equations abstract, preferring functional relationships over concrete formulas. This constitutes also a perspective for future research: derive concrete homogenized quantities for certain scaffold details, compare the outcome to experimental results, and employ the model in an optimization routine analogous to the one presented in \cite{poh2019optimization}. The $3$-dimensionality of the model makes an optimization of the scaffold porosity considerably more challenging from a numerical viewpoint -- but due to the efficient, homogenized, model it is within reach to provide patient specific optimal scaffold designs that depend on the individual's defect site and geometry, as well as their regeneration capacity.

	\subsection{The System of Equations} 
	Let $\Omega \subset \R^3$ be the domain of computation, i.e., the bone defect site, and let $I = [0,T]$ be some finite time interval. On the defect site we keep track of the local scaffold volume fraction called $\rho(x)$, with $x\in\Omega$. Equivalently, the relation to the local scaffold porosity $\theta$ is given by $\theta(x) = 1 - \rho(x)$, but we work with $\rho$ exclusively. Note that we do not assume a time dependency for $\rho$ as experimental findings \cite{pitt1981aliphatic} have shown that, in the time-window relevant for us, PCL degrades via bulk erosion. However, the molecular mass decreases and we keep track of this by introducing the exponential decay $\sigma (t) = e^{-k_1 t}$, making the product $\rho(x)\cdot\sigma(t)$ the quantity encoding the mechanical properties of PCL over time and space. Furthermore, we denote the local bone density by $b(t,x)$ and the three quantities $b, \sigma$ and $\rho$ together determine the mechanical material properties of the bone-scaffold composite. We model this composite in the linear elastic regime using an elastic tensor $\mathbb{C}(\rho,\sigma,b)$ to capture the material properties. 
	
	In the spirit of the homogenization approach we assume little on the concrete properties of this tensor, in particular we do not assume isotropy. For a particular choice of micro-structure $\mathbb{C}(\rho, \sigma, b)$ can be made explicit. In order to quantify the elastic stimulus throughout the bone-scaffold composite we introduce a displacement field $u(t,x)$ satisfying the equation of mechanical equilibrium \eqref{StrongElasticEquation}. The corresponding strain is denoted by $\varepsilon(u)$, with $\varepsilon(u) = \frac12 (Du + D^Tu)$ the symmetrized derivative. 
	
	For the biological environment we introduce $N$ bio-active molecules denoted by $a_1(t,x),\dots,a_N(t,x)$, these are endogenous angiogenic and osteoinductive factors which we assume to diffuse depending on the scaffold density $\rho$. This is captured by $D_i(\rho)$ in the equation \eqref{StrongDiffusionEquation} and is left as an abstract functional relationship for the same reasoning as the elastic tensor. Furthermore, we assume the bio-active molecules to decay at a certain rate and to be produced in the presence of strain and a local density of specific cells (e.g., osteoblasts) which we denote by $c(t,x)$. The essential quantity for the production of bio-active molecules is $|\varepsilon(u)|_\delta$, where $|\cdot|_\delta$ is a functional relationship which we propose to view as a usual Euclidean norm or a truncated version thereof, see also \eqref{absvalueDeltaEstimate}. The concentrations of bio-active molecules are normalized to unity in healthy tissue and the choice of decay and production rate should reflect this in a concrete simulation.
	
	Equation \eqref{StrongCellODE} governing the production of bone forming cells (here: osteoblasts) is modeled by logistic growth and a functional relationship $H(a_1,\dots,a_N,c,b)$ allowing driving factors for osteoblast production to be the concentrations of bio-active molecules (causing differentiation of stem cells to osteoblasts), the proliferation of osteoblasts and the maturity of the bone present. Note that we do not model diffusion in this equation as we assume that osteoblasts diffuse on a significantly lower level than the bio-active molecules. Of course, more than one cell type is present and responsible for bone growth. For simplicity we only include osteoblasts in this model, but an extension is easily feasible here.  Finally, the equation modelling bone growth \eqref{StrongODE} follows the same pattern as the one for osteoblast concentration. In summary, our system of equations reads

	\begin{align}
		0 &= \operatorname{div}\big(
		\mathbb{C}(\rho,\sigma,b)\varepsilon(u)
		\big) \label{StrongElasticEquation}
	    &\parbox{18em}{(mechanical equilibrium)}
		\\
		d_ta_i 
		&=
		\label{StrongDiffusionEquation}
		\operatorname{div}\big(
		D_i(\rho)\nabla a_i
		\big)
		+
		k_{2,i}|\varepsilon(u)|_\delta c
		-
		k_{3,i}a_i
		&\parbox{18em}{(diffusion, generation, and decay of $i=1\dots N$ bio-molecules)}
		\\
		d_tc 
		&=
		\label{StrongCellODE}
		H(a_1,\dots,a_N,c,b)\bigg(1 - \frac{c}{1-\rho}
		\bigg)
		&\parbox{18em}{(osteoblast generation)}
		\\
		d_tb
		&=		
		\label{StrongODE}
		K(a_1,\dots,a_N,c,b)\bigg(
		1-\frac{b}{1-\rho}
		\bigg)
		&\parbox{18em}{(bone regeneration driven by $a,b$ and $c$).}
	\end{align}
	In the above system $k_1,k_{2,i}, k_{3,i}\geq0$, $i=1,\dots,N$ are constants that need to be determined from experiments, compare to the section \ref{sectionNumericalExperiments} where we discuss certain choices. The functional relationships $\mathbb{C}, D_i(\rho), |\cdot|_\delta, H$ and $K$ are all required to satisfy certain technical assumptions that guarantee the well-posedness of the above system. We discuss this in detail in section \ref{sectionMathematicalFormulation}.

Finally, we need to specify boundary conditions. For the elastic equilibrium equation we allow mixed boundary conditions including the limiting cases of a pure displacement boundary condition and a pure stress boundary condition. As for the bio-active molecules we assume that these are in saturation, i.e., $a(t,x) = 1$ adjacent to bone and on the rest of the boundary of $\Omega$ we assume no-flux boundary conditions. For the initial time-point we propose $a_i(0,x) = a_{i,0} = 0$ inside of $\Omega$. This choice reflects the scenario of a scaffold that is not preseeded with exogenous growth factors. However, different choices of $a_{i,0}$ are admissible and allow the model to cover e.g., pre-seeding with osteoinductive factors. Finally, at the initial time we assume that no osteoblasts and no regenerated bone are present inside the domain of computation. In formulas, it holds for all $i = 1,\dots,N$
	\begin{align}
		a_i(0,x)\label{InitialMolecules}
		&=
		0
		&\parbox{18em}{for all $x\in\Omega$}
		\\
		a_i(t,x)\label{BoundaryDirichletMolecules}
		&=
		1
		&\parbox{18em}{for all $t\in I$, $x$ adjacent to bone}
		\\
		D_i^\rho\nabla a_i(t,x)\cdot\eta
		&=
		0
		&\parbox{18em}{for all $t\in I$, $x$ not adjacent to bone}
		\\ 
		\big(
		\mathbb{C}(\rho,\sigma,b)\varepsilon(u(t,x))
		\big)\cdot\eta
		&=
		g_N(x)
		&\parbox{18em}{on the Neumann boundary of $\Omega$}
		\\ 
		u(t,x)
		&=
		g_D(x)
		&\parbox{18em}{on the Dirichlet boundary of $\Omega$}
		\\ 
		c(0,x) = b(0,x)\label{InitialBone}
		&=
		0
		&\parbox{18em}{for all $x\in \Omega$.}
	\end{align}
	
	The model allows for a time dependent choice of the mechanical loading $g_D$ and $g_N$. Due to the long regeneration time horizon of approximately $12$ months, however, it is not expedient to resolve very short time-scales of, e.g., the mechanics of physical therapy. Instead, we consider suitably time-averaged loading conditions here.

	\subsection{Concrete Examples.}\label{SubsectionExamples}
	We provide a number of possibilities for choosing the functional relationships $\mathbb{C}, D, H$ and $K$ and boundary conditions for the mechanical equilibrium equation \ref{StrongElasticEquation}. For an easy example of the elastic tensor that does not need to be derived by a complicated homogenization procedure we simply use the Voigt bound. If we denote by $\mathcal{C}_b$ and $\mathcal{C}_\rho$ the elastic tensors of matured bone and intact PCL respectively (in their simplest form modelled as isotropic materials) we thus choose 
	\begin{align*}
		\mathbb{C}(\rho,\sigma,b) = b\mathcal{C}_b + \rho\sigma\mathcal{C}_\rho.
	\end{align*}
	This is in accordance with \cite{poh2019optimization} where the same idea was used in a model with only one spatial variable. Note that this $\mathbb{C}$ naturally is time-dependent as the quantities $b$ and $\sigma$ vary in time. While this example may serve as a first choice,  one could also fix a concrete scaffold micro-structure, such as a gyroid design, and derive the explicit homogenized material properties (see, e.g., \cite{allaire2012shape}).
	
	For the diffusivities $D_i(\rho)$ we propose a dependence on the scaffold density $\rho$, for example
	\begin{align*}
		D_i(\rho) = k_i(1-\rho)\operatorname{Id}
	\end{align*}
	where $k_i$ are constants that measure the diffusivity of the bio-active molecule $a_i$ without the presence of the scaffold $\rho$. The term $(1 - \rho)$ accounts for reduced diffusivity for high PCL volume fractions. It is heuristically clear, yet interesting to note, that a too dense scaffold impairs bone regeneration. This is reflected in our model through the diffusivity above, since the amount of bioactive molecules is linked to bone regeneration via the ODE \eqref{StrongODE}. One could also imagine to derive the tensor $D_i(\rho)$ through a homogenization process which would then again reflect the choice of a specific micro-structure. For mathematical well-posedness reasons we are unable to allow the diffusivity $D_i(\rho)$ to depend on the bone density $b$. Furthermore, we also assume that $D_i(\rho)$ does not depend on time.
   
	Finally, we consider the functional relationships $H$ and $K$ inducing the production and proliferation of osteoblasts and bone.
	To be covered by our mathematical analysis, in the realization of $H$ and $K$ not more than two of the bio-active molecules should be multiplied. This is a technical mathematical issue due to a possible lack of integrability. Compare also to Assumption \ref{AssumptionOnKAndH} where we discuss this issue rigorously. 
	Consequently, we provide an example involving two bio-active molecules $a_1$ and $a_2$. These can be assumed to have different production rates and half-lives. Then we set 
	\begin{align}\label{ExampleH1}
		H(a_1,a_2, c, b)
		=
		H(a_1, a_2, c)
		=
		k_6a_1a_2(1 + k_7c)
	\end{align}
	hence bone growth only takes place when the full bio-environment, i.e., both molecules $a_1$ and $a_2$ are present. Furthermore the proliferation of osteoblasts is represented by the term $(1 + k_7 c)$. Again $k_6$ and $k_7$ are some constants that need to be chosen in accordance with experiments.
	
	For $K$ we propose a similar equation, modelling that bone growth takes place given the presence of osteoblasts and a suitable biological environment, represented in the choice of $K$ through the factor $a_1$. More precisely we set
	\begin{align}\label{ExampleK1}
	    K(a_1,a_2, c, b) = K(a_1, c) = k_4a_1c.
	\end{align}
	
	Another choice for $K$ reflecting that different bio-active molecules are responsible for different stages of bone formation and maturation is possible. This makes the functional relationship dependent of $b$. We set
	\begin{align}\label{ExampleK2}
		K(a,b)
		=
		f_1(b)a_1c + f_2(b) a_2 c.
	\end{align}
	Now, $f_1$ can be chosen with support on small values of $b$, such that in this stage molecule $a_1$ is driving the growth, and $f_2$ with support on larger $b$, thus requiring $a_2$ in later stages of regeneration. We remark that empirically many different bio-molecules are observed and it is assumed that these are linked to different biological processes \cite{kempen2010growth}.
\section{Mathematical Formulation}\label{sectionMathematicalFormulation}
    In this section we describe the mathematical setting in which we prove the existence of a solution to the system of equations \eqref{StrongElasticEquation} -- \eqref{StrongODE}. We also state the assumptions the functional relationships $\mathbb{C}$, $D_i$,  $|\cdot|_{\delta}$ $H$ and $K$ are required to satisfy.
    
    \subsection{The Domain}\label{as:domain}
    Fix a time interval $I=[0,T]$ with $T>0$. The spatial domain $\Omega\subset\mathbb{R}^n$, with $n=1,2,3$ is assumed to be open, bounded and connected and for every equation we split the boundary $\partial\Omega$ into a Dirichlet part and a Neumann part. For the elastic equation we  write $\Gamma_D^e$ and $\Gamma_N^e$ for Dirichlet and Neumann boundary respectively, here $\Gamma_D^e = \emptyset$ is allowed. For the diffusion equations we write $\Gamma_D^d$ and $\Gamma_N^d$. To simplify notation we do not treat the case of different Dirichlet-Neumann partitions for different diffusion equations, though this does not lead to further mathematical complications. Finally we need to assume some regularity on $\Omega$ and the partition $\partial\Omega = \Gamma_D^d\cup\Gamma_N^d$ for the diffusion equations, namely the set $\Omega\cup\Gamma^d_N$ needs to be Gr\"oger regular which is a concept introduced in \cite{groger1989aw}, see also \cite{haller2009holder}. These regularity assumptions are tailored to provide a certain regularity of the solutions of the diffusion equations which we discuss in detail in Appendix \ref{sec:AppendixDiffusion}. These assumptions are very general and cover the cases one wants to use in practice.
    
    \subsection{Admissible Data}
    The admissible scaffold volume fractions $\rho$ are given as 
    \begin{equation}\label{AdmissableScaffolds} 
    P\coloneqq \{ \rho\in C^{0}(\overline{\Omega}) \mid c_P \leq \rho(x) \leq C_P \} 
    \end{equation}
    with some fixed constants $0 < c_P < C_P < 1$, excluding unreasonable scaffold designs.
    To a scaffold volume fraction $\rho\in P$ we assign the set $W_\rho$ of admissible cell and bone volume fractions, consisting of tuples of continuous functions in time and space
    \begin{equation}\label{AdmissableBones} 
    W_\rho \coloneqq \{ (c,b) \in C^0(\overline{I}\times\overline{\Omega})^2 \mid 0 \leq c(t,x), \ b(t,x) \leq 1 - \rho(x) \}. 
    \end{equation}
    
    \subsection{The Elastic Equation}
    
	We begin with the Hookean law $\mathbb{C}$. It depends on the scaffold and bone, i.e., on $\rho$, $\sigma$ and $b$ and varies therefore in space and time. We assume that the map 
	\begin{equation}\label{HookeanLaw}
	    W_\rho \to L^\infty(I, L^\infty(\Omega, \mathcal{L}(\mathcal{M}_s))) \quad\text{with}\quad b\mapsto(t\mapsto(x\mapsto\mathbb{C}(\rho,\sigma,b)(t,x)))
	\end{equation}
	is Lipschitz continuous with Lipschitz constant $L_\mathbb{C}$ independent of $\rho\in P$. Remember that $\sigma$ is a fixed exponential decay. Here $\mathcal{M}_s$ denotes the symmetric $n\times n$ matrices and $\mathcal{L}(\mathcal{M}_s)$ is the space of linear maps from $\mathcal{M}_s$ into itself, usually called the space of fourth order tensors. The space $L^\infty(I, L^\infty(\Omega, \mathcal{L}(\mathcal{M}_s)))$ denotes a Bochner space, i.e., a Banach-space valued Lebesgue space, see, e.g., \cite{diestel1977vector,boyer2012mathematical}. In the following we will often omit the cumbersome notation of dependencies on $x$ and $t$ for $\mathbb{C}$. Spelling out the definitions of the norms in \eqref{HookeanLaw} this Lipschitz continuity means that for all $M\in\mathcal{M}_s$ it holds
	\begin{equation}\label{HookeanLawLipEstimate}
	    \lvert \mathbb{C}(\rho(x),\sigma(t),b_1(t,x))M - \mathbb{C}(\rho(x),\sigma(t),b_2(t,x))M \rvert \leq L_{\mathbb{C}}\norm{b_1 - b_2}_{C^0}\lvert M \rvert
	\end{equation}
	for all $(c_1,b_1)$, $(c_2,b_2)\in W_\rho$ and uniformly in $\rho \in P$ and uniformly on the complement of a set of measure zero in $I\times\Omega$. Furthermore we assume that there are constants $0<c_{ \mathbb{C} } < \infty$ and $0 < C_{ \mathbb{C} } < \infty$ such that
	\begin{equation}\label{UpperBoundHookLaw}
	\sup_{\rho,c,b}\norm{\mathbb{C}(\rho,\sigma,b)}_{L^\infty(I, L^\infty(\Omega, \mathcal{L}(\mathcal{M}_s)))} \leq C_{\mathbb{C}} \ \ \text{and}\ \ \inf_{\rho, c, b}\mathbb{C}(\rho,\sigma,b)M:M \geq c_{\mathbb{C}}|M|^2
	\end{equation}
	where the supremum and infimum run over $\rho\in P$ and $b\in W_\rho$ and $A:B=\operatorname{tr} AB^T$ denotes the full contraction of matrices.
	We now discuss the weak formulation of equation \eqref{StrongElasticEquation}. Let $\rho\in P$ and $(c,b)\in W_\rho$ be some admissible functions. We first address the case where $\Gamma^e_D$ has non-vanishing measure and comment on the pure Neumann problem later. The strong form
	\[
	-\operatorname{div}\Big(\mathbb{C}(\rho,\sigma,b)\varepsilon(u)\Big) = 0 \ \text{in }\Omega, \quad u_{|\Gamma^e_D} = g_D^e, \quad \Big(\mathbb{C}(\rho,\sigma,b)\varepsilon(u)\Big)\eta_{|\Gamma^e_N} = g_N^e
	\]
	encodes that at every point in time mechanical equilibrium is achieved, making the equation time dependent. The function space for the weak formulation is: $L^2(I,H^{1,2}(\Omega,\mathbb{R}^n))$ with $H^{1,2}(\Omega,\mathbb{R}^n)$ being the Sobolev space of $\mathbb{R}^n$-valued, square integrable functions with square integrable derivatives, see for example \cite{brezis2010functional,grisvard2011elliptic,adams2003sobolev} for a detailed account of such spaces. If the context is clear, we will usually write $H^1(\Omega)$ instead of $H^{1,2}(\Omega,\mathbb{R}^n)$. The space of test functions is $L^2(I,H^1_{D_e}(\Omega))$, where $H^1_{D_e}(\Omega)$ is the subspace of $H^1(\Omega)$ whose members vanish on $\Gamma_D^e$. For the Dirichlet boundary values we require $g_D^e$ to be in $L^2(I,H^{1/2}(\Gamma^e_D,\mathbb{R}^n))$, with $H^{1/2}(\Gamma)$, for some $\Gamma\subset\partial\Omega$, being the trace space of $H^1(\Omega)$, see for example \cite{adams2003sobolev,grisvard2011elliptic}. The Neumann boundary values can be given as an element of $L^2(I,H^{1/2}(\Gamma_N^e,\mathbb{R}^{n})')$. Denoting by $\langle \cdot,\cdot \rangle_{H^{1/2}}$ the dual pairing of $H^{1/2}(\Gamma_N^e,\mathbb{R}^n)$ the weak formulation of \eqref{StrongElasticEquation} is
	\begin{align}
	    \int_I\int_\Omega \mathbb{C}(\rho,\sigma,b)\varepsilon(u):\varepsilon(\cdot)\,dxdt &= \int_I\langle g_N^e,\cdot \rangle_{H^{1/2}}\,dt \quad \text{in}\quad L^2(I,H^1_{D_e}(\Omega))' \label{WeakElasticity}
	    \\
	    u &= g_D^e \quad\text{in}\quad L^2(I, H^{1/2}(\Gamma_D^e)).\notag
	\end{align}
	The left hand side of \eqref{WeakElasticity} equation defines an operator
	\begin{equation*}
	    \mathcal{T}:L^2(I,H^1(\Omega)) \to L^2(I,H^1(\Omega))' \quad\text{with}\quad \mathcal{T}u = \int_I\int_\Omega \mathbb{C}(\rho,\sigma,b)\varepsilon(u):\varepsilon(\cdot)\,dxdt.
	\end{equation*}
	Note that the isometry $L^2(I,H^1(\Omega))'\, \Tilde{=}\, L^2(I,H^1(\Omega)')$ implies that the equation \eqref{WeakElasticity} can be understood to hold almost everywhere in time, which is precisely what we want for our model. Furthermore, Korn's inequality can be used to show that $\mathcal{T}$ is coercive, see \cite{ciarlet2010korn}. The advantage of the abstract formulation is that it makes the Lax-Milgram Lemma applicable. Now we comment on the pure Neumann boundary value problem, i.e., the case $\Gamma_N^e = \partial\Omega$. We define the spaces $W\coloneqq \ker(\varepsilon)\subset H^1(\Omega)$ and the quotient $H^1(\Omega)/W$. Note that $W$ consists of the functions of the form $w(x) = Ax + b$, where $A$ is an anti-symmetric matrix and $b\in\mathbb{R}^n$, see for example \cite{ciarlet1988mathematical}. For the pure Neumann problem consider the operator
	\begin{equation*}
	    \mathcal{T}: L^2(I,H^1(\Omega)/W) \to L^2(I,H^1(\Omega)/W)'
	\end{equation*}
	using the induced map $\hat{\varepsilon}:H^1(\Omega)/W\to L^2(\Omega, \mathcal{M}_s)$ in its definition
	\begin{equation*}
	    \mathcal{T}(u) = \int_I\int_\Omega \mathbb{C}(\rho,\sigma,b)\hat{\varepsilon}(u):\hat{\varepsilon}(\cdot)\,dxdt.
	\end{equation*}
	The codomain of this operator is $L^2(I,H^1(\Omega)/W)' \, \Tilde{=} \, L^2(I,(H^1(\Omega)/W)')$, which encodes a compatibility condition. We assume that our Neumann boundary condition is given as a function $g_N^e\in L^2(I, H^{1/2}(\partial\Omega)')$ that satisfies almost everywhere in $I$
	\begin{equation}\label{NeumannComaptibility}
	    \langle g_N^e(t),\cdot \rangle_{H^{1/2}} = 0 \quad\text{for all }w\in W.
	\end{equation}
	This guarantees that 
	\begin{equation*}
	    \int_I\langle g_N^e,\cdot\,\rangle_{ H^{1/2} }\,dt \in L^2(I,H^1(\Omega)/W)'
	\end{equation*}
	is an admissible right hand side. The pure Neumann problem consists then of finding $u\in L^2(I,H^1(\Omega)/W)$ such that
	\begin{equation*}
	    \int_I\int_\Omega \mathbb{C}(\rho,\sigma,b)\hat{\varepsilon}(u):\hat{\varepsilon}(\cdot)\,dxdt = \int_I\langle g_N^e,\cdot\,\rangle_{ H^{1/2} }\,dt \in L^2(I,H^1(\Omega)/W)',
	\end{equation*}
	Finally, let us remark that one can treat the Dirichlet, the Neumann and the mixed boundary value problem at once by always passing to the quotient $H^1_{D_e}(\Omega)/W$. In the case of a proper Dirichlet boundary condition we then have $W\cap H^1_{D_e}(\Omega) = \{0\}$, which implies $H^1_{D_e}(\Omega)/W = H^1_{D_e}(\Omega)$, hence recovers the Dirichlet or mixed case, and if $\Gamma_N^e = \partial\Omega$ we retrieve the pure Neumann case.
	
	\subsection{Diffusion Equations}
	Before we state the weak formulation of the diffusion equations, for the reader's convenience, we recall the concept of the time derivative we are using -- namely a regular Banach space valued distribution with a dense embedding $j\in\mathcal{L}(X,X')$ just as in \cite{boyer2012mathematical}. Let $(i,X,H)$ be a Gelfand triple, i.e., $X$ is a Banach space, $H$ is a Hilbert space and $i\in\mathcal{L}(X,H)$ has dense range. Then we set $j$ to be $j=i'\circ R\circ i$ where $R:H\to H'$ is the Riesz isometry and $i'$ denotes the Banach space adjoint of $i$. We say a function $a\in L^2(I,X)$ possesses a time derivative $d_ta\in L^2(I,X')$ if it holds
	\begin{equation*}
		\int_I (j\circ a)(t)\partial_t\varphi(t)\,dt = -\int_I d_ta(t)\varphi(t)\,dt \quad\forall\varphi\in\mathcal{D}(I).
	\end{equation*}
	The integrals are $X'$ valued Bochner integrals and we set $\mathcal{D}(I)\coloneqq C_c^\infty(I)$ as usual. This is used to define a generalized Sobolev space built on the triple $(i,X,H)$ as
	\begin{equation*}
		H^{1,2,2}(I,X,X') = \{ a\in L^2(I,X) \mid d_ta\in L^2(I,X') \}.
	\end{equation*}
	See in \cite[Chapter II, section 5]{boyer2012mathematical} for more information. We only remark that functions in this Sobolev space have representatives in $C^0(\overline{I},H)$, hence initial value problems can be formulated.
	
    To get to our concrete diffusion equations we let $\rho\in P$, $(c,b)\in W_\rho$ and, depending on the boundary conditions for the elastic equation, $u\in L^2(I,H^1(\Omega))$ or $u\in L^2(I,H^1(\Omega)/W)$ be some fixed functions. In order to work with homogeneous Dirichlet boundary conditions in space we write
	\begin{equation*}
		a_i(t) = \Tilde{a}_i(t) + 1 \quad\text{with}\quad \Tilde{a}_i(t) \in H^1_{D_d}(\Omega) \text{ for } i=1,\dots,N.
	\end{equation*}
	Here $H^1_{D_d}(\Omega)$ denotes the subspace of $H^1(\Omega)$ with vanishing trace on $\Gamma_D^d$.
	We can thus seek $\Tilde{a}_i$ in the space $H^{1,2,2}(I,H^1_{D_d}(\Omega),H^1_{D_d}(\Omega)')$ built around the triple $(\operatorname{id}_{|H^1_{D_d}},H^1_{D_d},L^2)$ satisfying the equation
	\begin{align}
	\int\langle d_t\Tilde{a}_i,\cdot \rangle_{H^1_{D_d}} + \iint D_i^\rho\nabla \Tilde{a}_i \nabla\cdot + k^3_i\Tilde{a}_i\cdot\,dxdt &= \iint (k^2_i|\varepsilon(u)|_\delta c - k^3_i)\,\cdot\,dxdt
	\\ 
	\Tilde{a}_i(0) &= -1.
	\end{align}
	The first equation is an equality in the space $L^2(I,H^1_{D_d}(\Omega))'$, i.e., it is required to hold when tested with all members of $L^2(I,H^1_{D_d}(\Omega))$. In the second equation, the initial conditions is an equality in the space $L^2(\Omega)$. For every $i=1,\dots,N$ we have different constants $k^2_i$ and $k^3_i$ and also different diffusivities $D_i^\rho$. Note that the quantity $|\varepsilon(u)|_delta$ is well defined, even though the solution of the elastic equation is only unique up to rigid body motions. We assume furthermore that the $D_i^\rho$ are time-independent, measurable, essentially bounded and coercive, precisely 
	\begin{align}
		D_i^\rho \in L^\infty(\Omega,\mathcal{M}_s)\label{Diff_reg}
		\\
		\langle D_i^\rho\xi,\xi \rangle \geq c_D|\xi|^2 \quad \forall \xi\in\mathbb{R}^n \label{Diff_Coerciv}
	\end{align}
	where $\mathcal{M}_s$ again denotes the symmetric $n\times n$ matrices and the inequality in \eqref{Diff_Coerciv} is to be understood uniformly in $x\in\Omega$, $\rho\in P$ and $i=1,\dots,N$. Finally the function $|\cdot|_\delta:\mathbb{R}^{n\times n}\to[0,\infty)$ is required to to be globally Lipschitz and to satisfy an estimate of the form
	\begin{equation}\label{absvalueDeltaEstimate}
	    |A|_\delta \leq C_1|A| + C_2 \quad\text{for all }A\in\mathbb{R}^{n\times n}
	\end{equation}
	where $C_1,C_2 >0$ and $|A|$ denotes the Euclidean norm of a matrix.
	
	\subsection{Ordinary Differential Equations}
	We treat the ordinary differential equations in the vector valued sense and focus here on the cell equation \eqref{StrongCellODE}, the bone equation \eqref{StrongODE} being treated analogously. For each $x\in\Omega$, we thus seek a function $c_{x}$ satisfying the ODE
	\begin{equation*}
	    c'_{x}(t) = H(a_1(t,x),\dots,a_N(t,x), c_{x}(t), b(t,x))\bigg( 1 + \frac{c_{x}(t)}{1-\rho(x)} \bigg)
	\end{equation*}
	with $c_{x}(0) = 0$. If there is a solution for all $x\in\Omega$ we obtain a function $c$ in time and space, i.e., $c:I\times\Omega\to\mathbb{R}$ with $c(t,x)\coloneqq c_x(t)$.
	As $H(a_1,\dots,a_N,c,b)$ can not generally assumed to be continuous, a reasonable space to work in is
	\begin{equation*}
		W^{1,p}(I,X) = \{ c\in L^p(I,X) \mid d_tc \in L^p(I,X) \},
	\end{equation*}
	similar to the space for the diffusion equation, but without the identification $j:X\hookrightarrow X'$. An existence and uniqueness result in this setting can be found in the appendix, see Theorem \ref{LocalExistenceTheorem}. 
	
	In our concrete case we choose $X = C^0(\overline{\Omega})$, $p=2$, so 
	for fixed $\rho\in P$ and $a=(a_1,\dots,a_N)\in H^{1,2,2}(I,H^1(\Omega),H^1_{D_d}(\Omega)')^N$ we seek $c\in W^{1,2}(I,C^0(\overline{\Omega}))$ satisfying
	\begin{equation}\label{CellODE}
	    d_tc = H(a_1,\dots,a_N,c,b)\bigg( 1 - \frac{c}{1 - \rho} \bigg)\quad\text{with}\quad c(0) = 0.
	\end{equation}
	We assume that  $H$ is a Nemytskii operator induced by a function which we again denote by $H$,
    \begin{equation}
	H:\mathbb{R}^{N+2}\to\mathbb{R}\quad\text{with}\quad (a_1,\dots,a_N)=a\mapsto H(a)
	\end{equation}
	such that $H(a,c,b)\geq0$ whenever $a_1,\dots,a_N,b,c\geq0$. Furthermore we assume that $H$ is locally Lipschitz continuous. Note that by some abuse of notation we denote by $a$, $b$ and $c$ both a function in a Sobolev space and a vector in Euclidean space.
	
	For the bone ODE we work in the same space and seek $b\in W^{1,q}(I,C^0(\overline{\Omega}))$ satisfying
	\begin{equation}\label{ODE}
		d_tb = K(a_1,\dots,a_N,c,b) \bigg( 1 - \frac{b}{1-\rho} \bigg) \quad\text{with}\quad b(0) = 0.
	\end{equation}
	We assume the functional relationship $K$ is induced by
	\[
	    K:\mathbb{R}^{N+2} \to \mathbb{R} \quad\text{with} \quad (a,b,c) = (a_1,\dots,a_N,b,c)\mapsto K(a,b,c)
	\] 
	that satisfies $K(a_1,\dots,a_N,b,c)\geq0$ for $a_1,\dots,a_N,b,c\geq0$ and that $K$ is locally Lipschitz continuous as a map $K:\mathbb{R}^{N+2}\to\mathbb{R}$. Finally, we need another assumption on $H$ and $K$ that is connected to the integrability and the regularity properties of the solutions to the diffusion equations, see assumption \ref{AssumptionOnKAndH}. 
	We summarize our setting.
	\begin{assumption}\label{assumptionSetting}
	    We assume domain regularity as discussed in subsection \ref{as:domain}, define the admissible scaffold densities $P$ in \eqref{AdmissableScaffolds} and the set $W_\rho$ in \eqref{AdmissableBones}. The material tensor $\mathbb{C}$ satisfies \eqref{HookeanLaw} and \eqref{UpperBoundHookLaw} and admissible boundary conditions for the elastic equation are given in \eqref{WeakElasticity} and \eqref{NeumannComaptibility}. For the diffusion we assume \eqref{Diff_reg} and \eqref{Diff_Coerciv} and $|\cdot|_\delta$ must satisfy \eqref{absvalueDeltaEstimate}. The functional relationships $H$ and $K$ need to be locally Lipschitz, preserve positivity and satisfy the technical assumption \ref{AssumptionOnKAndH} concerning integrability.
	\end{assumption}

\section{Existence and Uniqueness}\label{sec:existence}
In this section we will prove that there exists a unique solution to the system \eqref{StrongElasticEquation}--\eqref{StrongODE} in the weak sense, i.e., there are functions $u^* = \Tilde{u}^* + u_{g_D^e}$ with $\Tilde{u}^* \in L^2(I,H^1_{D_e}(\Omega)/W)$ and $u_{g_D^e|\Gamma_D^e} = g_D^e$, $a^* = \Tilde{a}^* + 1$ with $\Tilde{a}^* \in H^1(I,H^1_{D_d}(\Omega), H^1_{D_d}(\Omega)')$, $c^* \in W^{1,p}(I,C^0(\overline{\Omega}))$ and $b^* \in W^{1,q}(I,C^0(\overline{\Omega}))$ satisfying 
\begin{gather}\label{WeakEquationsBegin}
    \int_I\int_\Omega \mathbb{C}(\rho,\sigma,b^*)\hat{\varepsilon}(\Tilde{u}^* + u_{g_D^e}) : \hat{\varepsilon}(\cdot)\,dxdt = \int_I\langle g_N^e,\cdot \, \rangle_{H^{1/2}(\Gamma^e_N)}dt 
    \\
    \int\langle d_t\Tilde{a}^*_i,\cdot \rangle + \iint D_i^\rho\nabla \Tilde{a}^*_i \nabla\cdot + k^3_i\Tilde{a}^*_i\cdot\,dxdt = \iint (k^2_i|\varepsilon(u^*)|_\delta c^*-k^3_i)\,\cdot\,dxdt
	\\ 
	\Tilde{a}^*_i(0) = -1, \quad\text{with } i = 1,\dots,N,
	\\
	d_tc^* = H(a^*_1,\dots,a^*_N,c^*,b^*)\bigg( 1 - \frac{c^*}{1 - \rho} \bigg)\quad\text{with}\quad c^*(0) = 0,
	\\
	d_tb^* = K(a^*_1,\dots,a^*_N,c^*,b^*) \bigg( 1 - \frac{b^*}{1-\rho} \bigg) \quad\text{with}\quad b^*(0) = 0.\label{WeakEquationsEnd}
\end{gather}
The proof of this result relies essentially on the elementary fixed point theorem of Banach which we will employ for the complete metric space $W_\rho$. The strategy is to fix $\rho\in P$, then start with some arbitrary admissible functions $(c,b)\in W_\rho$ and to solve the equations successively. More precisely, the elastic equation will yield $u = u(c,b)$, the diffusion equations $a_i = a_i(c, u)$, the cell equation will be solved with data $a_i$ and $b$ yielding an updated cell function $\overline{c} = \overline{c}(a_i, b)$ and finally the bone equation will be solved with data $a_i$ and $c$ to get an updated bone function $\overline{b} = \overline{b}(a_i, c)$. This procedure gives rise to an operator $\mathcal{I}$ which we will refer to as the iteration operator, formally
\[
    \mathcal{I}:W_\rho \to W_\rho \quad\text{with}\quad (c,b)\mapsto (\overline{c},\overline{b}).
\]
It is easy to see that all possible solutions to \eqref{WeakEquationsBegin}--\eqref{WeakEquationsEnd} correspond to all possible fixed-points of $\mathcal{I}$. The crucial part of the proof consists of establishing regularity for the solutions of the diffusion equations, see also Appendix \ref{sec:AppendixDiffusion} for a discussion of results known in the literature serving our purpose.

Finally, the whole strategy discussed above does only work on a short time interval $I = [0,T]$, i.e., $T$ small enough. However, by a continuation argument we can afterwards extend this solution to span any finite time interval. We will need a technical assumption on the ODEs in connection with the iteration operator $\mathcal{I}$. This is due to the fact that we cannot guarantee an $L^\infty(I\times\Omega)$ bound on the solutions to the diffusion equations. See also Remark \ref{RemarkAssumptionODE} on when the following assumption holds.
\begin{assumption}\label{AssumptionOnKAndH}
    Let $\rho\in P$ and $(c,b)\in W_\rho$ and denote by $u \in L^2(I,H^1(\Omega)/W)$, $a \in L^2(I, C^0(\overline{\Omega})^N)$ and $\overline{c} \in C^0(\overline{I} \times \overline{\Omega})$ the functions produced by solving the equations successively as in the definition of the iteration operator $\mathcal{I}$. The existence and regularity of these solutions is discussed in the main theorem. Assume there exists $p\in[1,\infty]$ such that for every bounded set $B\subset C^0(\overline{\Omega})$ there are functions $m_B^H \in L^p(I)$ and $L^H_B \in L^1(I)$ such that it holds
    \begin{gather}\label{AssumptionH1}
        \norm{ H(a(t), \Tilde{c}, b(t) }_{C^0(\overline{\Omega})} \leq m_B^H(t) \quad \text{for all } \Tilde{c} \in B, \ \text{a.e.\ }in \ I,
        \\
        \label{AssumptionH2}
        \norm{H(a(t),\Tilde{c}, b(t)) - H(a(t),\Tilde{\Tilde{c}}, b(t))}_{C^0(\overline{\Omega})} \leq L^H_B(t) \text{ for all } \Tilde{c}, \Tilde{\Tilde{c}} \in B, \ \text{a.e.\ }in \ I.
    \end{gather}
    Furthermore assume that there are functions $\lambda^H$ and $\beta^H$ in $L^1(I)$ and $\alpha^H\in L^2(I)$ such that we can estimate, independently of the choice of $(c,b) \in W_\rho$ (and consequently a and $\overline{c}$), 
    \begin{align}\label{AssumptionH3}
        \norm{ H(a(t), \overline{c}(t), b(t)) }_{{C^0(\overline{\Omega})}} \leq \lambda^H(t).
    \end{align}
    Additionally, uniformly for any $(c_1, b_1), (c_2, b_2) \in W_\rho$ and corresponding $a^1, a^2, \overline{c}_1, \overline{c}_2$, we have
    \begin{align}\label{AssumptionH4}
        \norm{ H(a^1(s),b_1(s),\overline{c}_1(s)) - H(a^2(s), b_2(s), \overline{c}_2(s))}_{C^0(\overline{\Omega})} \notag
        \leq \beta^H(t)\norm{\overline{c}_1(s) - \overline{c}_2(s)}_{C^0(\overline{\Omega})} \\
        + \alpha^H(t)\Big[ \norm{a^1(s) - a^2(s)}_{C^0(\overline{\Omega})^N} + \norm{b_1(s) - b_2(s)}_{C^0(\overline{\Omega})} \Big].
    \end{align}
    For $K$ we assume analogous properties, i.e., there is $q \in [1, \infty]$ such that for every bounded $B\subset C^0(\overline{\Omega})$ there are $m^K_B \in L^q(I)$ and $L^K_B \in L^1(I)$ as well as $\lambda^K$, $\beta^K\in L^1(I)$ and $\alpha^K\in L^2(I)$ satisfying estimates as above. 
\end{assumption}

	\begin{theorem}[Existence \& Uniqueness]\label{ExistenceThm}
		Let $\rho \in P$ be fixed and let the Assumptions \ref{assumptionSetting} and \ref{AssumptionOnKAndH} be fulfilled. Then there exist unique functions $u^* = \Tilde{u}^* + u_{g_D^e}$ with $\Tilde{u}^* \in L^2(I,H^1_{D_e}(\Omega)/W)$ and $u_{g_D^e|\Gamma_D^e} = g_D^e$, $a^* = \Tilde{a}^* + 1$ with $\Tilde{a}^* \in H^1(I,H^1_{D_d}(\Omega), H^1_{D_d}(\Omega)')$, $c^* \in W^{1,p}(I,C^0(\overline{\Omega}))$ and $b^* \in W^{1,q}(I,C^0(\overline{\Omega}))$ solving the system \eqref{WeakEquationsBegin} -- \eqref{WeakEquationsEnd}.
		\begin{proof}
			We need to establish the contraction and self-mapping property of $\mathcal{I}$. 
			Let us thus fix two tuples $(c_1,b_1)$ and $(c_2,b_2) \in W_\rho$. We aim to show an estimate of the form
			\begin{align*} 
			    \norm{\mathcal{I}(c_1,b_1) - \mathcal{I}(c_2,b_2)}_{C^0(\overline{I}\times\overline{\Omega})^2} &= \norm{\overline{c}_1 - \overline{c}_2 }_{C^0(\overline{I}\times\overline{\Omega})} +  \norm{\overline{b}_1 - \overline{b}_2}_{C^0(\overline{I}\times\overline{\Omega})}
			    \\&\leq
			    C(I)\big(\norm{c_1 - c_2}_{C^0(\overline{I}\times\overline{\Omega})} + \norm{b_1 - b_2}_{C^0(\overline{I}\times\overline{\Omega})}\big),
			\end{align*}
			where $C(I)\to 0$ with $|I|\to 0$, making $\mathcal{I}$ the desired self-mapping for $T$ small enough.
			\ \\ \ \\
			\emph{The Elastic Equation.}
			We will treat a pure Neumann and a mixed boundary value problem simultaneously. We endow the space $H^1_{D_e}(\Omega)/W$ with the norm $\norm{\varepsilon(\cdot)}_{L^2(\Omega)}$, which by Korn's inequality is equivalent to the natural one on $H^1_{D_e}(\Omega)/W$, see for example \cite{ciarlet2010korn}. By definition of $g_D^e$, there is a function $u_{g_D^e}\in L^2(I,H^1(\Omega)/W)$ such that $u_{g_D^e|\Gamma_D^e} = g_D^e$. In the weak formulation of the elastic equation we seek $\Tilde{u}_i\in L^2(I, H^1_{D_e}(\Omega)/W)$ satisfying
			\begin{equation}\label{WeakElasticEquationNew}
			    \int_I\int_\Omega \mathbb{C}(\rho,\sigma,b_i)\hat{\varepsilon}(\Tilde{u}_i + u_{g_D^e}) : \hat{\varepsilon}(\cdot)\,dxdt = \int_I\langle g_N^e,\cdot \, \rangle_{H^{1/2}(\Gamma^e_N)}dt 
			\end{equation}
			in the space $L^2(I,H_{D_e}^1(\Omega)/W)'$. Then $u_i\coloneqq \Tilde{u}_i + u_{g_D^e}$ is the solution we are interested in. Note that if $\Gamma_D^e$ has vanishing measure, we can choose $u_{g_D^e} = 0$ and $H^1_{D_e}(\Omega)/W = H^1(\Omega)/W$. On the other hand, if $\Gamma_D^e$ has positive measure, then $W\cap H^1_{D_e}(\Omega) = 0$ and $H^1_{D_e}(\Omega)/W = H^1_{D_e}(\Omega)$. The equation \eqref{WeakElasticEquationNew} leads to the operators
			\begin{equation*}
			    \mathcal{T}_{b_i}: L^2(I, H^1_{D_e}(\Omega)/W) \to L^2(I,H^1_{D_e}(\Omega)/W)'
			\end{equation*}
			with
			\begin{equation*}
			    \mathcal{T}_{b_i}v = \int_I\int_\Omega \mathbb{C}(\rho,\sigma,b_i)\hat{\varepsilon}(v):\hat{\varepsilon}(\cdot)\,dxdt
			\end{equation*}
			and right hand sides
			\begin{equation*}
			    f_{b_i} = \underbrace{\int_I\langle g_N^e, \,\cdot \, \rangle_{H^{1/2}(\Gamma_N^e)}dt}_{=:f^N} - \underbrace{\int_I\int_\Omega \mathbb{C}(\rho,\sigma,b_i)\hat{\varepsilon}(u_{g_D^e}):\hat{\varepsilon}(\cdot)\,dxdt}_{=:f^D_{b_i}}.
			\end{equation*}
			By our assumption \eqref{UpperBoundHookLaw} and Korn's inequality the operators $\mathcal{T}_{b_i}$ are coercive with coercivity constant $c_\mathbb{C}$. Applying the Lax-Milgram Lemma we find that there are unique solutions $\Tilde{u}_1$ and $\Tilde{u}_2\in L^2(I, H^1_{D_e}(\Omega)/W)$ to $\mathcal{T}_{b_i}\Tilde{u}_i = f_{b_i}$. By the duality
			\begin{equation*}
			    L^2(I, H^1(\Omega)/W)' = L^2(I, (H^1(\Omega)/W)')
			\end{equation*}
			we know that almost everywhere in $I$ the function $u_i(t) = \Tilde{u}_i(t) + u_{g_D^e}(t)$ satisfies
			\begin{align*}
			    \int_\Omega \mathbb{C}(\rho,\sigma,b_i)(t)\hat{\varepsilon}(\Tilde{u}_i)(t):\hat{\varepsilon}(\cdot)\,dx &= \langle g_N^e(t) \rangle_{H^{1/2}(\Gamma^e_N)}
			    \\&-
			    \int_\Omega \mathbb{C}(\rho,\sigma,b_i)\hat{\varepsilon}(u_{g_D^e}(t)):\hat{\varepsilon}(\cdot)\,dx
			\end{align*}
			in the space $H^1(\Omega)/W$. Using Lax-Milgram again we get using the boundedness and coercivity constants from \eqref{UpperBoundHookLaw}
			\begin{equation*}
			    \norm{u_i(t)}_{H^1(\Omega)/W} \leq c_{\mathbb{C}}^{-1}\Big[ \norm{g_N^e(t)}_{H^{1/2}(\Gamma^e_N)'} + C_{\mathbb{C}}\norm{u_{g_D^e}(t)}_{H^1(\Omega)/W} \Big].
			\end{equation*}
			As the above estimate is independent of $\rho, c_i$ and $b_i$ it holds
			\begin{align}
			    \sup_{\rho,c,b}\norm{u(\rho,b)}_{L^2(I,H^1(\Omega)/W)} &\leq C(I), \label{L2BoundU}
			\end{align}
			where $u(\rho,c,b)$ denotes the solution of the elastic problem to the data $\rho\in P$ and $(c,b)\in W_\rho$. To show that $C(I)$ tends to zero with $|I|\to 0$ we employ the dominated convergence theorem of Lebesgue. Finally we come back to estimate the difference $u_1 - u_2$. We claim that 
			\begin{equation}\label{DifferenceInU}
			    \norm{u_1 - u_2}_{L^2(I,H^1(\Omega)/W)} \leq C(I)\norm{b_1 - b_2}_{C^0(\overline{I}\times \overline{\Omega})}
			\end{equation}
			where again $C(I)\to 0$ with $|I|\to 0$. To establish this, note that $\Tilde{u}_1-\Tilde{u}_2 = u_1 - u_2$ and compute
			\begin{align*}
			    f_{b_1} - f_{b_2} = f_{b_1}^D - f_{b_2}^D &= \mathcal{T}_{b_1}\Tilde{u}_2 - \mathcal{T}_{b_2}\Tilde{u}_2
			    =
			    \mathcal{T}_{b_1}(\Tilde{u}_1 - \Tilde{u}_2) + \mathcal{T}_{b_1}\Tilde{u}_2 - \mathcal{T}_{b_1}\Tilde{u}_2 - \mathcal{T}_{b_2}\Tilde{u}_2.
			\end{align*}
			Hence $\mathcal{T}_{b_1}(u_1 - u_2) = (\mathcal{T}_{b_2}\Tilde{u}_2 - \mathcal{T}_{b_1}\Tilde{u}_2 ) + ( f_{b_1}^D - f_{b_2}^D )$ and using $\norm{\mathcal{T}_{b_1}^{-1}}\leq c_{\mathbb{C}}^{-1}$ we find
			\begin{align*}
			    \norm{u_1 - u_2}_{L^2(I,H^1(\Omega)/W)} &\leq c_{ \mathbb{C} }^{-1}\norm{\mathcal{T}_{b_2}\Tilde{u}_2 - \mathcal{T}_{b_1}\Tilde{u}_2}_{L^2(I, H^1(\Omega)/W)'}
			    \\&+
			    c_{ \mathbb{C} }^{-1}\norm{ f_{b_1}^D - f_{b_2}^D }_{L^2(I, H^1(\Omega)/W)'}.
			\end{align*}
			We estimate the terms of the right hand side using the Lipschitz continuity of $\mathbb{C}$ which we assumed in \eqref{HookeanLawLipEstimate}, combining it with \eqref{L2BoundU} to find
			\begin{align*}
			    \norm{\mathcal{T}_{b_2}\Tilde{u}_2 - \mathcal{T}_{b_1}\Tilde{u}_2}_{L^2(I, H^1(\Omega)/W)'} &\leq L_{ \mathbb{C} }\norm{b_1 - b_2}_{C^0(\overline{I}\times\overline{\Omega})}\norm{\Tilde{u}_2}_{L^2(I,H^1(\Omega)/W)}
			    \\ &\leq
			    L_{ \mathbb{C} }C(I)\norm{b_1 - b_2}_{C^0(\overline{I}\times\overline{\Omega})}
			    \\ & =
			    C(I)\norm{b_1 - b_2}_{C^0(\overline{I}\times\overline{\Omega})}
			\end{align*}
			and
			\begin{align*}
			    \norm{ f_{b_1}^D - f_{b_2}^D }_{L^2(I, H^1(\Omega)/W)'} &\leq L_{ \mathbb{C} }\norm{u_{g^e_D}}_{L^2(I,H^1(\Omega)/W)}\norm{b_1 - b_2}_{C^0(\overline{I}\times\overline{\Omega})}
			    \\&=
			    C(I)\norm{b_1 - b_2}_{C^0(\overline{I}\times\overline{\Omega})}.
			\end{align*}
			\ \\ \ \\
			\noindent\emph{The Diffusion Equations.} Given the functions $c_i, u_i$ with $i=1,2$ and $\rho$, we turn to the diffusion equations. We seek functions $a^i=\Tilde{a}^i+1$ where the $\Tilde{a}^i$ are members of $H^{1}(I,H^1_{D_d}(\Omega),H^1_{D_d}(\Omega)')^N$, that means $a^i=(a^i_1,\dots,a^i_N) = (\Tilde{a}^i_1 + 1,\dots,\Tilde{a}^i_N + 1)$, $i=1,2$, denoting the components of $a^i$ with lower indices. For $j=1,\dots, N$ the $\Tilde{a}^i_j$ are sought to satisfy the following equation in $L^2(I, H^1_{D_d}(\Omega))'$
			\begin{align*}
			\underbrace{\int_{I}\langle d_t\Tilde{a}^i_j,\cdot \, \rangle_{H^1_D}\,dt}_{=:d_t\Tilde{a}^i_j}
			+
			\underbrace{\iint D^\rho_j\nabla \Tilde{a}^i_j \nabla\cdot\, + k^3_j\Tilde{a}_j^i\,\cdot\,dxdt}_{=:\mathcal{M}_j(\rho)\Tilde{a}_j^i}
			=
			\underbrace{\iint (k^2_j|\varepsilon(u_i)|c_i-k^3_j)\,\cdot\,dxdt}_{=:f^j_{u_i,c_i}}
			\end{align*}
			and initial value $\Tilde{a}_j(0)=-1$ in $L^2(\Omega)$. The operators
			\begin{align*}
				(d_t+\mathcal{M}_j(\rho),\operatorname{ev}_0):
				H^{1}(I,H^1_{D_d}(\Omega),H^{1}_{D_d}(\Omega)')
				\to
				L^2(I,H^1_{D_d}(\Omega))'\times L^2(\Omega)
			\end{align*}
			are linear homeomorphisms, see for example \cite{ern2013theory} for a proof, which essentially relies on the coercivity of $\mathcal{M}_j(\rho)$. This explains why we assumed \eqref{Diff_Coerciv} and hence we can guarantee the existence of the $\Tilde{a}^i_j.$ We state two important properties of the solutions $a^i_j$ and their differences $a^1_j - a^2_j$, to which references or proofs can be found in the Appendix \ref{sec:AppendixDiffusion}. The first is a lower pointwise bound, it holds for $j=1,\dots,N$ and $i=1,2$
			\begin{align}\label{APointwiseBound}
				0 \leq 1 + \Tilde{a}^i_j(t,x) = a^i_j(t,x)
				\quad\text{almost everywhere\,\,in }I\times\Omega.
			\end{align}
			This is due to the positivity of the right hand sides $f^j_{u_i,c_i}$.  Secondly, we look at the equations satisfied by the differences $a^1_j - a^2_j$. These equations possess right hand sides $f^j_{u_1,c_1} - f^j_{u_2,c_2}$ in $L^2(I,L^2(\Omega))$ and with $(a^1_j - a^2_j)(0) = 0$ smooth initial conditions. Then, using regularity for mixed boundary value problems, see Theorem \ref{L2CAlphaRegularity}, there is $\alpha > 0$ such that
			\begin{align}\label{ParabolicH2Regularity}
			    \norm{a^1_j - a^2_j}_{L^2(I,C^\alpha(\overline{\Omega}))} \leq C \norm{f^j_{u_1,c_1} - f^j_{u_2,c_2}}_{L^2(I,L^2(\Omega))}.
			\end{align}
			The constant $C$ is uniform in the data $\rho\in P$, $(c,b)\in W_\rho$ and $u(\rho)$. We claim now that we get the following estimate for the difference $a^1 - a^2$
			\begin{align}\label{EstimateDiffusionDifference}
				\norm{a^1-a^2}_{L^2(I,C^\alpha(\overline{\Omega}))^N}
				\leq
				C \Big( \norm{c_1 - c_2}_{C^0(\overline{I}\times\overline{\Omega})} + \norm{u_1-u_2}_{L^2(I,H^1/W)}\Big)
			\end{align}
			with $C$ not blowing up as $|I|\to 0$. This estimate is obtained, using \eqref{ParabolicH2Regularity} and estimating the difference $f^j_{u_1,c_1} - f^j_{u_2,c_2}$. It holds
			\begin{equation*}
			f^j_{c_1,u_1} - f^j_{c_2,u_2} = k_2^j|\varepsilon(u_1)|_\delta(c_1 - c_2) + k^j_2(|\varepsilon(u_1)|_\delta - |\varepsilon(u_2)|_\delta)c_2.
			\end{equation*}
			Using the fact that $c_1$ takes values in the unit interval and the assumptions on $|\cdot|_\delta$, see \ref{absvalueDeltaEstimate}, it follows
			\begin{align*}
			    \norm{f^j_{c_1,u_1} - f^j_{c_2,u_2}}_{L^2(I,L^2(\Omega))} &\leq C\big( \norm{\varepsilon(u_1)}_{L^2(I,L^2(\Omega))} + 1\big) \norm{c_1 - c_2}_{C^0(\overline{I}\times\overline{\Omega})}
			    \\&+ 
			    C\norm{\varepsilon(u_1 - u_2)}_{L^2(I,L^2(\Omega))}.
			\end{align*}
			Invoking \eqref{L2BoundU} we know that $\norm{\varepsilon(u_1)}_{L^2(I,L^2(\Omega))}$ is bounded uniformly in the data $\rho\in P$ and $(c,b)\in W_\rho$. Combining this with the identity
			\begin{equation*}
			    \norm{\varepsilon(u_1 - u_2)}_{L^2(I,L^2(\Omega))} = \norm{u_1 - u_2}_{L^2(I,H^1/W)}
			\end{equation*}
			we conclude.
			\ \\ \ \\
			\noindent\emph{The Cell ODE.}
		    We turn now to the Cell ODE and solve this equation twice, once with data $\rho,a_1^1,\dots, a_N^1$ and $b_1$, producing a function $\overline{c}_1$, and once with $\rho, a_1^2,\dots, a_N^2$ and $b_2$ yielding $\overline{c}_2$. The solutions $\overline{c}_1$ and $\overline{c}_2$ are members of the space $W^{1,p}(I,C^0(\overline{\Omega}))$ and consequently of $C^0(\overline{I}\times\overline{\Omega})$ satisfying $0\leq \overline{c}(t,x)\leq 1-\rho(x)$ solving the ODE
		    \begin{equation*}
			d_t\overline{c}_i = H(a^i_1,\dots,a^i_N,b_i,\overline{c}_i)\bigg( 1 - \frac{\overline{c}_i}{1 - \rho} \bigg) \quad\text{with}\quad \overline{c}_i(0) = 0.
			\end{equation*}
			These facts are proven as Lemma \ref{SolveCellODE} in the Appendix. Our goal is to estimate the difference $\overline{c}_1 - \overline{c}_2$ and we claim that it holds
			\begin{equation}\label{EstimateCellODE}
			    \norm{\overline{c}_1 - \overline{c}_2}_{C^0(\overline{I}\times\overline{\Omega})} \leq C(I)\Big( \norm{a^1 - a^2}_{L^2(I,C^\alpha(\overline{\Omega})^N)} + \norm{b_1 - b_2}_{C^0(\overline{I}\times\overline{\Omega})} \Big)
			\end{equation}
			where $C(I)$ tends to zero with $|I|\to 0$. To prove the estimate \ref{EstimateCellODE} we use the fundamental theorem of the space $W^{1,p}(I,C^0(\overline{\Omega}))$ and write 
			\begin{align*}
			    \overline{c}_i(t) &= \int_0^t \underbrace{ H(a^i(s),b_i(s), \overline{c}_i(s))\bigg( 1 - \frac{\overline{c}_i(s)}{1 - \rho} \bigg)}_{=:\gamma_i(s)}\,ds.
			\end{align*}
			By Assumption \ref{AssumptionOnKAndH} the expression $\gamma_1(s) - \gamma_2(s)$ can be estimated
			\begin{align*}
			    \norm{\gamma_1(s) - \gamma_2(s)}_{C^0(\overline{\Omega})} &\leq \norm{ H(a^1, b_1, \overline{c}_1)(s) - H(a^2,b_2, \overline{c}_2)(s) }_{C^0(\overline{\Omega})} \norm{1 - \frac{\overline{c}_1(s)}{1 - \rho}}_{C^0(\overline{\Omega})} 
			    \\&+
			    \norm{H(a^2,b_2, \overline{c}_2)(s)}_{C^0(\overline{\Omega})}\norm{(1 - \rho)^{-1}}\,\norm{\overline{c}_1(s) - \overline{c}_2(s)}_{C^0(\overline{\Omega})}
			    \\&\leq
			    C\alpha^H(s)\Big[ \norm{a^1(s) - a^2(s)}_{C^0(\overline{\Omega})^N} + \norm{b_1(s) - b_2(s)}_{C^0(\overline{\Omega})} \Big]
			    \\&+
			    C(\beta^H(s) + \lambda^H(s))\norm{\overline{c}_1(s) - \overline{c}_2(s)}_{C^0(\overline{\Omega})}.
			\end{align*}
			Now we can apply Gr\"onwall's and H\"older's inequality, using that $\alpha^H\in L^2(I)$ and $\beta^H, \lambda^H \in L^1(I)$ to obtain
			\begin{align*}
			    \norm{\overline{c}_1(t) - \overline{c}_2(t)}_{C^0(\overline{\Omega})} &\leq C \int_I \alpha^H(s) \Big[ \norm{ a^1(s)- a^2(s) }_{C^0(\overline{\Omega})^N} + \norm{ b_1(s) - b_2(s)}_{C^0(\overline{\Omega})} \Big]
			    \\&\leq
			    C\norm{\alpha^H}_{L^2(I)}\Big[  \norm{a_1 - a_2 }_{L^2(I,C^0(\overline{\Omega})^N)} + \norm{b_1 - b_2}_{L^2(I,C^0(\overline{\Omega}))} \Big]
			    \\&\leq
			    C(I)\Big[ \norm{a^1 - a^2}_{ L^2(I,C^\alpha(\overline{\Omega}))} + \norm{b_1 - b_2}_{C^0(\overline{I}\times\overline{\Omega})} \Big].
			\end{align*}
			Here we used that $\norm{\alpha^H}_{L^2(I)}\to 0$ with $|I|\to 0$, which follows from Lebesgue's dominated convergence theorem.
			As the right side of the estimate is independent of $t\in I$, this shows that \eqref{EstimateCellODE} holds. 
			\ \\ \ \\
			\noindent\emph{The Bone ODE.} Finally we treat the Bone ODE. Again we solve it twice, with data $a^i_1,\dots,a^i_N, c_i$ and $\rho$ producing $\overline{b}_i$ with $i=1,2$. The functions $\overline{b}_1$ \& $\overline{b}_2$ are members of $W^{1,q}(I,C^0(\overline{\Omega}))$ and consequently of $C^0(\overline{I}\times\overline{\Omega})$ satisfying $0\leq \overline{b}_i(t,x)\leq 1 - \rho(x)$ and 
			\begin{equation*}
			    d_t\overline{b}_i = K(a^i_1,\dots a_N^i,\overline{b}_i,c_i)\bigg( 1 - \frac{\overline{b}_i}{1 - \rho} \bigg)
			\end{equation*}
			This means that $\overline{b}_i\in W_\rho$, hence making the iteration map $\mathcal{I}$ a self mapping. All these properties are established as in the case of the Cell ODE. Repeating our computations for $\overline{c}_1 - \overline{c}_2$ we find
			\begin{align}\label{EstimateBoneAndCellODE}
				\norm{\overline{b}_1-\overline{b}_2}_{C^0(\overline{I}\times\overline{\Omega})}
				\leq
				C(I) \Big( \norm{a^1-a^2}_{L^2(I,C^\alpha(\overline{\Omega})^N)} + \norm{c_1-c_2}_{C^0(\overline{I}\times\overline{\Omega})} \Big).
			\end{align}
			and the constant $C(I)$ tends to zero as $|I|\to 0$. 
			\ \\ \ \\
			\noindent\emph{Contraction Property of $\mathcal{I}$.} We collect all estimates to see that $\mathcal{I}$ is a contractive self-mapping for $|I|$ small enough. Use \eqref{EstimateBoneAndCellODE}, \eqref{EstimateCellODE}, \eqref{EstimateDiffusionDifference}, and \eqref{DifferenceInU} to conclude
			\begin{align*}
			\norm{\overline{b}_1-\overline{b}_2}_{C^0(\overline{I}\times\overline{\Omega})}
			&\leq
			C(I)\Big( \norm{a^1 - a^2}_{L^2(I,C^\alpha(\overline{\Omega})^N)} + \norm{c_1 - c_2}_{C^0(\overline{I}\times\overline{\Omega})} \Big)
			\\&\leq
		    C(I)\Big( \norm{u^1 - u^2}_{L^2(I,H^1(\Omega)/W)} + \norm{c_1 - c_2}_{C^0(\overline{I}\times\overline{\Omega})} \Big)
			\\
			&\leq
			C(I)\Big( \norm{b_1 - b_2}_{C^0(\overline{I}\times\overline{\Omega})} + \norm{c_1 - c_2}_{C^0(\overline{I}\times\overline{\Omega})}  \Big).
			\end{align*}
			and the estimate for $\norm{\overline{c}_1-\overline{c}_2}_{C^0(\overline{I}\times\overline{\Omega})}$ works identically. Consequently, it holds
			\begin{align*}
			    \norm{\mathcal{I}(c_1,b_1) - \mathcal{I}(c_2,b_2)}_{C^0(\overline{I}\times \overline{\Omega})^2} \leq C(I)\Big[\norm{c_1 - c_2}_{C^0(\overline{I}\times \overline{\Omega})} + \norm{b_1 - b_2}_{C^0(\overline{I}\times \overline{\Omega})}\Big]
			\end{align*}
			As $C(I)\to 0$ with $|I|\to 0$, the contraction map principle implies that $\mathcal{I}:(c,b)\mapsto (\overline{c},\overline{b})$ possesses a unique fix point for $|I|$ small enough.
			\ \\ \ \\
	\noindent\emph{Long-Time Existence.} We established the existence of a solution $(u^*,a^*,c^*,b^*)$ on an interval $[0,T]$ where $T>0$ is chosen to make $\mathcal{I}$ a contraction. Now we use the well defined functions $c^*(T,\,\cdot\,)$ and $b^*(T,\,\cdot\,)\in C^0(\overline{\Omega})$ as initial data for the ODEs and as $a^*\in C^0([0,T],L^2(\Omega)^N)$ the function $a^*(T)\in L^2(\Omega)^N$ serves as start value for the diffusion equations. Repeating the computations we find that there exists a unique solution $(u^{**},a^{**},c^{**},b^{**})$ to the system on the interval $[T-\varepsilon,2T-\varepsilon]$ for some small $\varepsilon>0$. On the overlap $[T-\varepsilon,T]$ the solutions $(u^{**},a^{**},c^{**},b^{**})$ and $(u^*,a^*,c^*,b^*)$ agree and thus we found the unique solution on the interval $[0,2T-\varepsilon]$. As $T$ does not depend on the initial values of neither $a^*$, $c^*$ nor $b^*$ this iterates to span every finite time interval.
	\end{proof}
	\end{theorem}
	\begin{remark}\label{RemarkAssumptionODE}
	    We discuss the validity of Assumption \ref{AssumptionOnKAndH}, which is given in an implicit form. We treat roughly two cases. Either, only treating pure Dirichlet problems for the diffusion equations or assuming strong properties for $|\cdot|_\delta$, one can establish an $L^\infty(I \times \Omega)$ bound on the solutions of the diffusion equations and can then apply Lemma \ref{howToHandleAssumption33}, or one is allowed to only multiply at most two different components of $a$ in order not to violate the integrability. Furthermore we will always assume that we are in the setting of section \ref{sectionMathematicalFormulation}.
	    \begin{itemize}
	        \item [(i)] Assume that $|\cdot|_\delta:\mathbb{R}^{n \times n}\to \mathbb{R}$ is a bounded function. Then the solution to the diffusion equations lie in $L^\infty(I \times \Omega)$ with a bound on the uniform norm not depending on $\rho \in P$ and $b \in W_\rho$. In view of Lemma \ref{howToHandleAssumption33} this implies that Assumption \ref{AssumptionOnKAndH} holds.
	        \item[(ii)] Assume that we consider a pure Dirichlet problem for the diffusion equations and that the Dirichlet data on the parabolic boundary lies in the space $L^\infty(I,L^\infty(\partial\Omega))$. Theorem 7.1 and Corollary 7.1 in \cite{ladyzhenskaia1968linear} show that the solutions of the diffusion equations are members of $L^\infty(I\times\Omega)$ with a uniform bound on their norms. Here one crucially needs the Dirichlet information on the parabolic boundary, thus the assumptions. We currently do not know whether a similar result is available in the case of mixed boundary conditions.
	        \item[(iii)] If we do not assume anything besides the setting of section \ref{sectionMathematicalFormulation}, in the choice of $H$ and $K$ not more than two of the components of $a$ should be multiplied. In particular the choice $|\cdot|_\delta = |\cdot|$ is covered.  A concrete example is given in section \ref{sectionNumericalExperiments}. 
	    \end{itemize}
	\end{remark}
	\begin{lemma}\label{howToHandleAssumption33}
	    Let Assumption \ref{assumptionSetting} hold and assume that for any choice of $\rho$ and $b\in W_\rho$ the function $a\in H^1(I,H^1_{D_d}(\Omega), H^1_{D_d}(\Omega)')^N$ produced by the iteration operator $\mathcal{I}$ is a member of $L^\infty(I\times \Omega)^N$ with a bound on the $L^\infty(I \times \Omega)$ norm that is uniform in $\rho\in P$ and $b\in W_\rho$. Then Assumption \ref{AssumptionOnKAndH} is satisfied.
	    \begin{proof}
	        For a fixed bounded set $B\subset C^0(\overline{\Omega})$ the following subset of $\mathbb{R}^{N + 2}$
	        \[
	            \{ ( a(t,x), \Tilde{c}(t,x), b(t,x) ) \mid \Tilde{c} \in B, \ (t,x)\in I\times\Omega \}
	        \]
	        is relatively compact. By the continuity of $H$ we can choose $m_B^H$ to be a constant function, i.e., a member of $L^\infty(I)$. Using the Lipschitz continuity of $H$ on the set defined above, we are able to establish property \eqref{AssumptionH2}. Now let $b_1$ and $b_2\in W_\rho$ be given and correspondingly $a^1, a^2, c^1$ and $c^2$. Using that the set 
	        \[
	            \{ ( a^i(t,x), c_i(t,x), b_i(t,x) ) \mid i = 1,2 \text{ and }(t,x) \in I \times \Omega \}
	        \]
	        is bounded independently of $\rho, b^1, b^2$ etc.\,\,we may use again the above reasoning and obtain that \eqref{AssumptionOnKAndH} holds with a constant function $\lambda^H$ and \eqref{AssumptionOnKAndH} with constant functions $\beta^H$ and $\alpha^H$. The remaining requirements in assumption \ref{AssumptionOnKAndH} are satisfied likewise.
	    \end{proof}
	\end{lemma}
\section{Numerical Experiments}\label{sectionNumericalExperiments}
In \cite{cipitria2015bmp} porous PCL scaffolds with a periodic honeycomb structure and $87\%$ porosity were used as a treatment strategy for $30\text{mm}$ tibial defects in an ovine model. This experiment was conducted in two groups, one preseeding the scaffold with a special bio-active molecule (BMP) and the second group without such preseeding. Here, we aim to numerically recreate the experiment without preseeding, using a concrete instance of our computational model.

As usual, the experimental setup in \cite{cipitria2015bmp} includes the use of a so-called fixateur -- a titanium or steel plate that is fixed to the bone surrounding the defect site using screws. This fixateur is used to provide additional mechanical stability. We include this device in a simplistic manner in our simulations, neglecting the effect of screws. From a modeling perspective, the fixateur acts as a stress shield on one side of the defect and thus influences bone growth significantly.

 As a concrete instance of our model we use two bioactive molecules and consider the following system of equations
\begin{align*}
    0 &= \operatorname{div}\Big( \mathbb{C}(\rho,\sigma, b)\varepsilon(u) \Big)
    \\
    d_ta_1 &= \operatorname{div}\Big(  D(\rho) \nabla a_1 \Big) + k_{2,1}|\varepsilon(u)|c - k_{3,1}a_1
    \\
    d_ta_2 &= \operatorname{div}\Big(  D(\rho) \nabla a_2 \Big) + k_{2,2}|\varepsilon(u)|c - k_{3,2}a_2
    \\
    d_tc &= k_6a_1a_2(1 + k_7c)\bigg( 1 - \frac{c}{1 - \rho} \bigg)
    \\
    d_tb &= k_4a_1c\bigg( 1 - \frac{b}{1 - \rho} \bigg).
\end{align*}
We use mixed boundary values for the elastic equilibrium equation, with a surface traction stemming from a force of $0.3$kN on the top of the cylinder in the  model with fixateur. The bottom of the computational domain is assumed to be fixed, i.e., subjected to zero Dirichlet boundary conditions and the remaining part of the boundary is subject to zero stress boundary conditions. These boundary conditions are chosen to represent the maximal stress that repeatedly occurs, having an ovine model in mind, where a specimen can weigh between 45--160kg. For a healthy individual without bone defect, we assume a force of $2.25$kN. This difference is important as it will influence the choice of the generation and decay rate of the bio-active molecules that are normalized for healthy bone.
For the bio-active molecules we assume that they are present in saturation, i.e. $a_1(t,x) = a_2(t,x) = 1$, adjacent to bone and otherwise we assume a non-flux boundary condition. Osteoblast and bone density is set to zero at the initial time-point. Note that the concrete choice of boundary conditions here should be considered a proof of concept. Further, more detailed numerical studies are forthcoming. 
\subsection{Model Parameters}
We report the choices for the constants and functional relationships in table \ref{tableConstants}. Some comments are in order.
\begin{itemize}
    \item [(a)] In a healthy individual, given appropriate clinical interventions, bone defects should be completely bridged with low to medium weight-bearing capacity after 6 months, see \cite{zimmermann2010trauma}. The bone remodelling process to follow can take 3 to 5 years until the full function of the bone is restored. We therefore consider a time span of 12 months for our model, which we identified as the critical phase for scaffold mediated bone healing.
    \item [(b)] The PCL decay parameter, $k_1$, is based on the experimental studies in \cite{pitt1981aliphatic}, which shows that after one year $30\%$ of the molecular mass remains. 
    \item [(c)] The surface traction is set to $0.001$gPa corresponding to a force of $0.3$kN over a surface of $300\operatorname{mm}^2$. We propose to view this time-constant surface traction as an averaged maximal stress. Furthermore we assume that due to the injury this averaged maximal stress is considerably lower than what is to be expected in a healthy individual, where we set it to $0.0075$gPa corresponding to the aforementioned $2.25$kN.
    \item [(d)] The constants $k_{2,i}, k_{3,i}$, $i = 1,2$ governing the generation and decay of bioactive molecules are difficult to obtain from the literature compare for example to the discussion in \cite{poh2019optimization}. The values for $k_{3,1}$ and $k_{3,2}$ correspond to a half-life of 31 and 62 hours respectively and are chosen to achieve a realistic model outcome. Consequently generation rate constants $k_{2,1}$ and $k_{2,2}$ are chosen such that a surface traction of $0.0075$gPa -- corresponding to a force of $2.25$ kN over a surface of $300\textrm{mm}^2$ -- results in an equilibrium state for $a_1 = a_2 = 1$ when $c = 1$, that is when the concentration of osteoblast equals that of healthy bone.
    \item [(e)] The diffusivity $D(\rho) = k_5(1 - \rho)$ is controlled by the porosity $1 - \rho$ and the constant $k_5$. With $k_5 = 260 \textrm{mm}^2/\textrm{month}$ we set it to a standard value for the diffusion of bioactive molecules that is measured for soluble proteins, see \cite{badugu2012digit} and \cite{yu2009fgf8}.
    \item [(f)] We use Voigt's bound as an approximation of the material properties of the bone-scaffold composite. More precisely, we model bone and PCL as linear isotropic materials with material constants chosen as collected in Table \ref{tableConstants}. The effective properties of the compositum are then obtained by adding the weighted tensors.
    \item [(g)] The constant $k_4$ drives the rate of bone regeneration, $k_6$ is related to the overall osteoblast production and $k_7$ influences the effect of osteoblast proliferation. These values are fitted to achieve realistic outcome in the simulations. 
\end{itemize}
\begin{table}[h!]
\centering
\caption{Parameters for the bone regeneration model}
\label{tableConstants}
 \begin{tabular}{| c | c | c |} 
 \hline
 \rowcolor{lightgray}
 \textbf{Param.} & \textbf{Value} & \textbf{Description} \\ [0.5ex] 
 \hline\hline
 $T$ & $12$ months & Period of bone regeneration \\ 
 \hline
 $\Omega$ & $L = 30$mm, $r = 10$mm & Cylinder with length $L$, radius $r$ \\
 \hline
 $\rho$ & $\rho \equiv 0.13 $ & Scaffold volume fraction \\
 \hline
 $\mathbb{C}(\rho,\sigma,b)$ & $b\mathcal{C}_b + \rho\sigma\mathcal{C}_\rho$ & Voigt bound for composites\\
 \hline
 $D(\rho)$ & $k_5(1 - \rho)$ & Diffusivity of bioactive molecules\\
 \hline
 $(\lambda_b,\mu_b)$ & (2.88GPa, 1.92GPa) & Derived from $(E_b,\nu_b) = (5\operatorname{GPa}, 0.3)$ \\
 \hline
 $(\lambda_\rho, \mu_\rho)$ & (1.97GPa, 0.17GPa) & Derived from $(E_\rho,\nu_\rho) = (0.5\operatorname{GPa}, 0.46)$ \\
 \hline
 $\mathcal{C}_b$ & $\mathcal{C}_b A = \lambda_b\operatorname{tr}(A)\operatorname{Id} + 2\mu_b A$ & Material tensor of healthy bone \\
 \hline
 $\mathcal{C}_\rho$ & $\mathcal{C}_\rho A = \lambda_\rho\operatorname{tr}(A)\operatorname{Id} + 2\mu_\rho A$ & Material tensor of PCL \\
 \hline
 $k_{1}$ & $0.1$ per month & PCL absorbation rate constant \\
 \hline
 $k_{2,1}$ & $10500$ & Generation rate first molecule \\
 \hline
 $k_{2,2}$ & $5250$ & Generation rate second molecule \\
 \hline
 $k_{3,1}$ & $16$ & Decay rate first molecule \\
 \hline
 $k_{3,2}$ & $8$ & Decay rate second molecule \\
 \hline
 $k_{4}$ & $0.2$ & Bone regeneration constant \\
 \hline
 $k_5$ & $260\operatorname{mm}^2/\operatorname{month}$ & Diffusivity of the $a_i$ w/o scaffold \\
 \hline
 $k_{6}$ & $0.5$ & Osteoblast generation constant \\
 \hline
 $k_{7}$ & $0.07$ & Proliferation constant for osteoblasts \\
 \hline
\end{tabular}
\end{table}
\begin{figure}
    \centering
    \includegraphics[width=12.5cm]{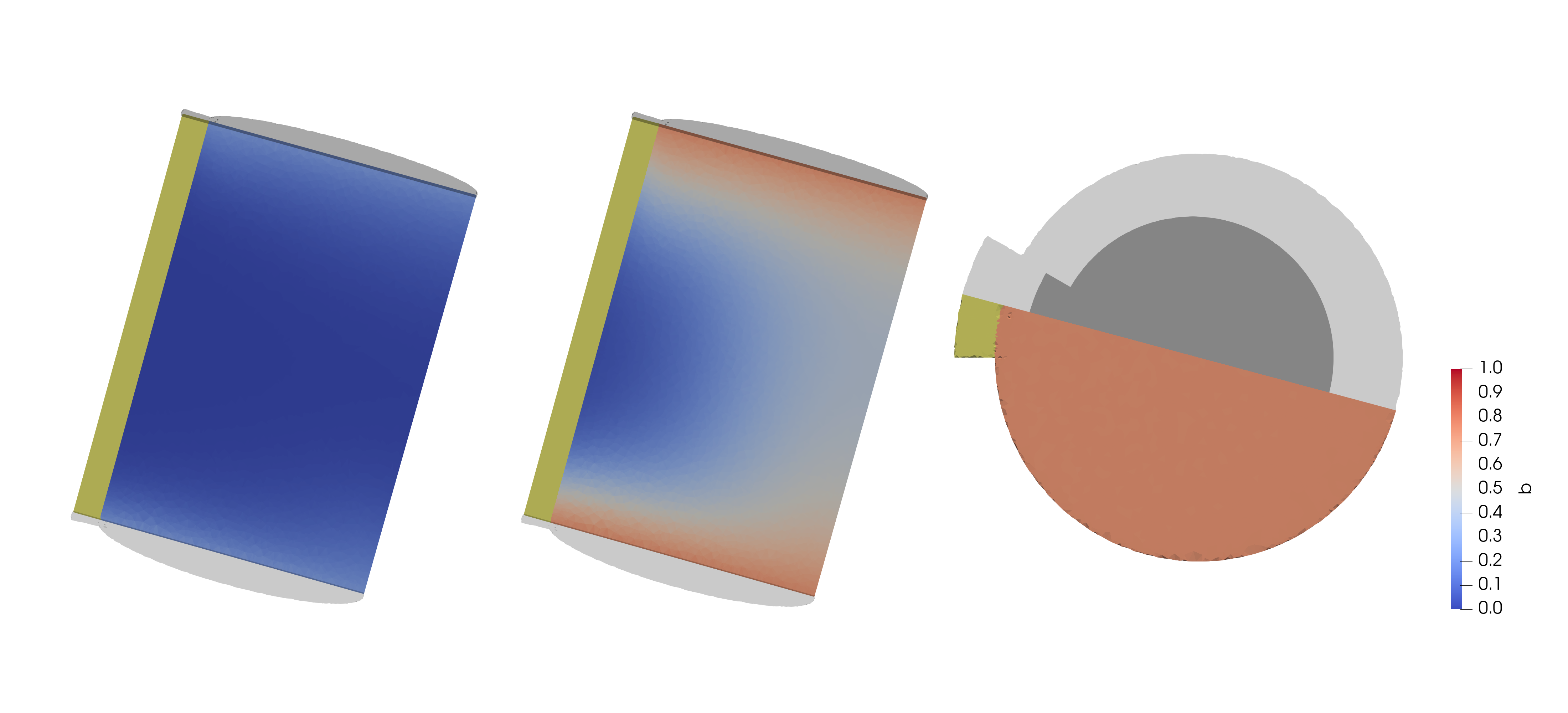}
    \caption{Experiment including fixateur. Shown is a vertical section through the cylindrical defect site. Fixateur domain is colored in gold. From left to right: regenerated bone at 3 months, 12 months and a view on top of the defect site. The grey colored areas illustrate the top and bottom cylinder/fixateur caps.}
    \label{Fixateur3and12Months}
\end{figure}
\begin{figure}
    \centering
    \includegraphics[width=12.5cm]{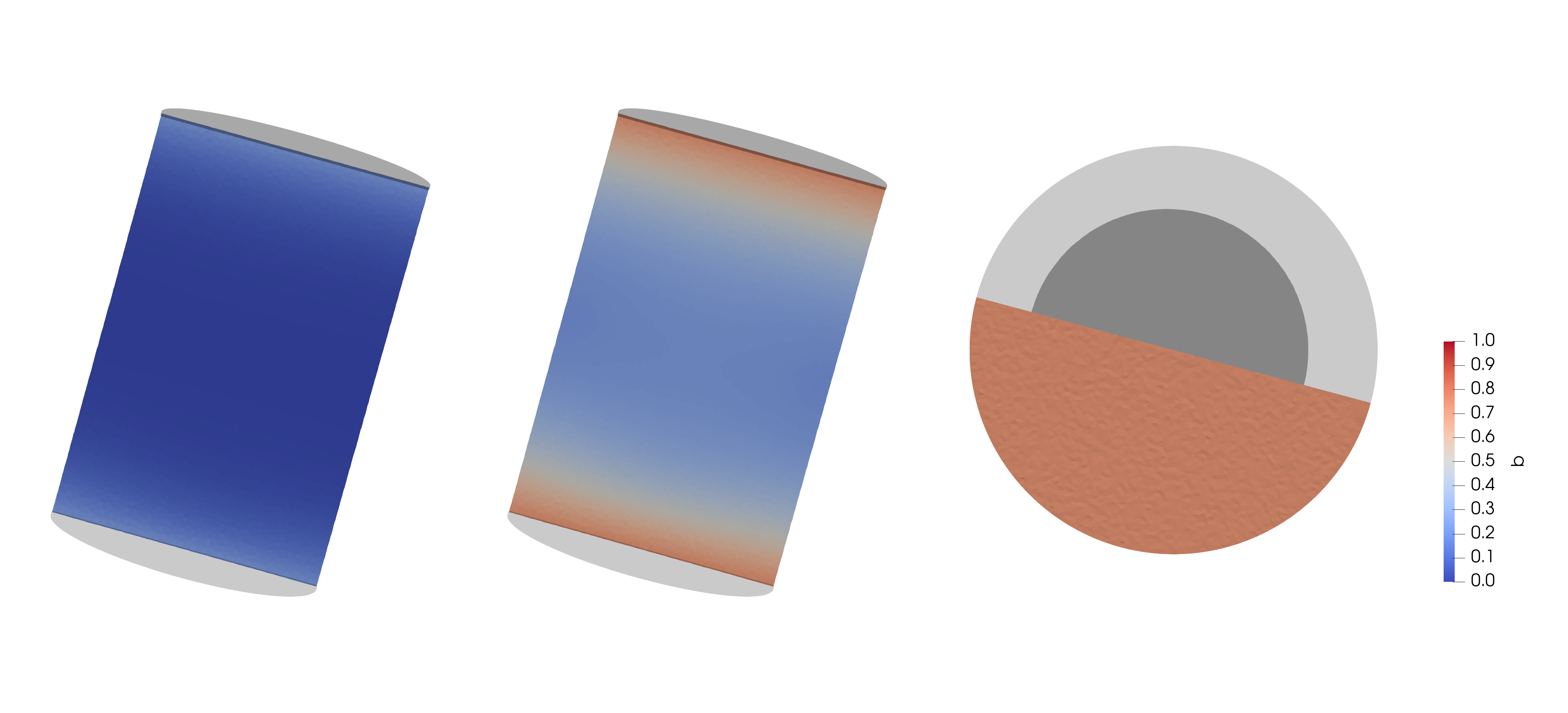}
    \caption{Experiment excludig fixateur. Shown is a vertical section through the cylindrical defect site. From left to right: regenerated bone at 3 months, 12 months and a view on top of the defect site. The grey colored areas illustrate the top and bottom cylinder caps.}
    \label{Cylinder3and12Months}
\end{figure}
\subsection{Numerical Implementations}
We use a simple first-order implicit in time Euler scheme to solve the equations displayed in the order displayed above. The fact that an implicit approach is feasible is due to the simple structure of the ODEs and the linearity of the diffusion equation. It is worth mentioning that this reduces the computational cost of solving the system drastically as only very few time steps are needed to achieve acceptable accuracy in the simulations. The elastic and the diffusion equation are discretized using P1 elements and the  meshes were generated using the Computational Geometry Algorithms Library CGAL \cite{boissonnat2000triangulations}. 

\subsection{Discussion of Numerical Simulations}

In Figure \ref{Fixateur3and12Months} the domain of computation with an added fixateur is shown. Here we assume the material of the fixateur to be titanium with Young's modulus chosen to 100GPa and a Poisson's ratio of 0.31. Bone growth and osteoblast production is disabled 
\begin{wrapfigure}{l}{0.5\linewidth}
    \includegraphics[width=0.95\linewidth]{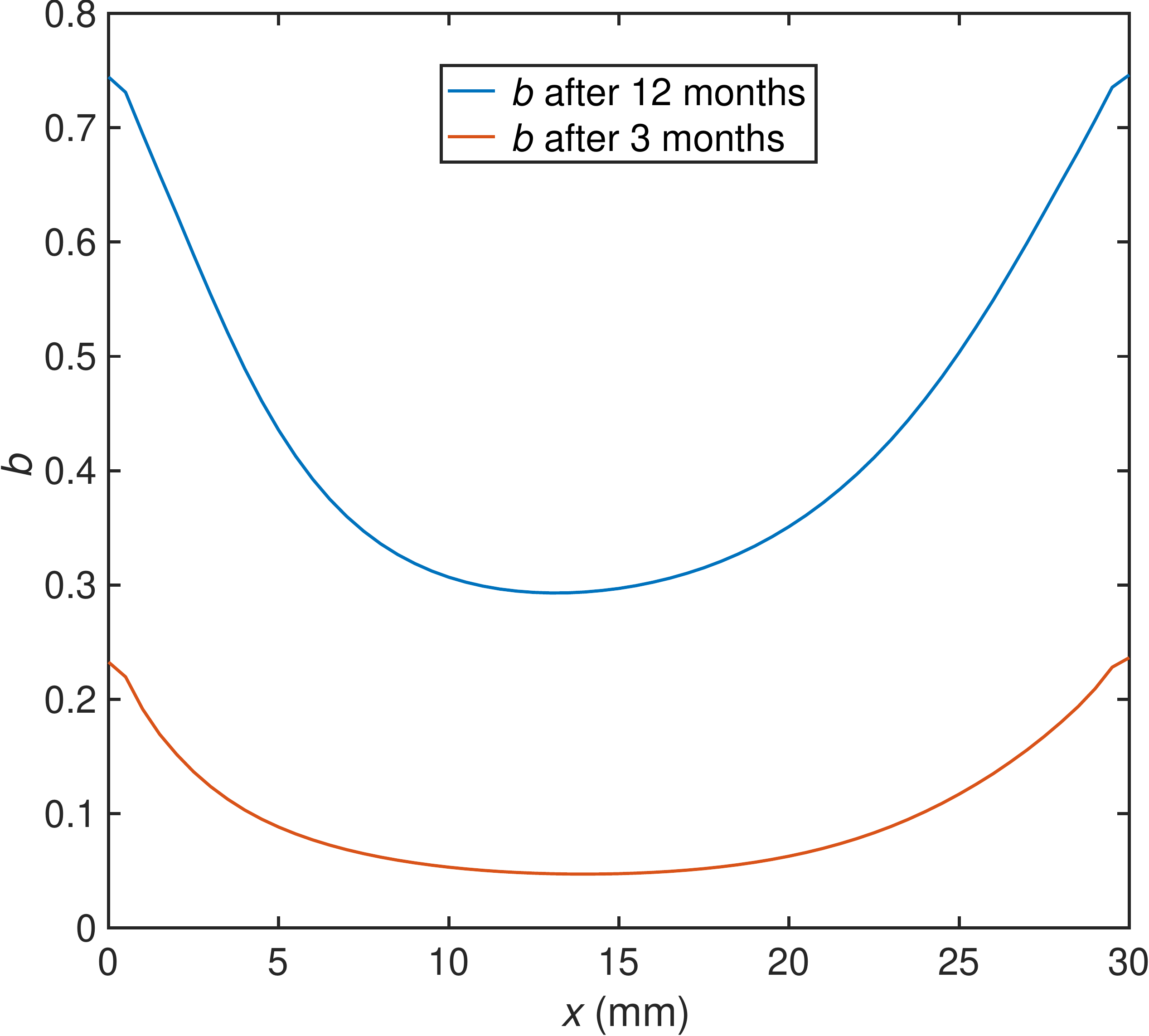}
    \caption{Relative bone density averaged over horizontal slices after $3$ and $12$ months in the experiment including the fixateur.}
    \label{RegeneratedBoneSlices}
\end{wrapfigure}
  in the space occupied by the fixateur. In Figure \ref{RegeneratedBoneSlices} we present the relative bone density averaged over horizontal slices in the fixateur experiment at $3$ and $12$ months. We observe that both the regenerated bone after $3$ and after $12$ months agree well with the experimental results shown in \cite[Figure 2, `Scaffold only']{cipitria2015bmp}. There, the same shape of regenerated bone, with a flat area in the middle of the defect site and a significant gradient towards the proximal and distal interface, is observed.

In Figure \ref{Fixateur3and12Months}, the result of the stress shielding effect of the fixateur is clearly visible, with little regenerated bone in the central part of the defect site close to the fixateur. Comparing to \cite[Figs 4C, 5C]{viateau07} or \cite[Figures 3a, 3b]{reichert11}  we see that this is also observed in  experiments. Bone mass loss due to stress shielding is indeed a long recognized, major issue in orthopaedic surgery \cite{Schwyzer1985, Terjesen09}.

The computation excluding the fixateur is performed using a reduced surface traction that is set to $70\%$ of the surface traction in the fixateur model to account for the stress shielding of the fixateur. This experiment is the direct analogon of the 1D model in \cite{poh2019optimization}. Naturally, we see that bone regenerates symmetrically and that the result is essentially a one dimensional distribution of bone comparable to the results in \cite{poh2019optimization}. Note that the asymmetries encountered in the more realistic model including the fixateur can not be resolved by a one-dimensional simplification. This has important implications for the porosity optimization of scaffolds where a three dimensional simulation can thus help to achieve a more appropriate optimal design.

\section*{Acknowledgements}
The  authors  gratefully  acknowledge   support  from   BMBF within the e:Med program in the SyMBoD consortium (grant number 01ZX1910C). Furthermore the authors thank Luca Courte (Freiburg) and Dorothee Knees (Kassel) for helpful suggestions and discussions.

\appendix

\section{Properties of Diffusion Equation}\label{sec:AppendixDiffusion}
This section provides the regularity results needed in the existence proof of Theorem \ref{ExistenceThm}. We begin by stating the regularity results in section \ref{app:regul}, then discuss the notion of Gr\"oger regular sets in section \ref{app:groeger} and conclude with the proofs in section \ref{subsectionProofs}.
\subsection{Regularity Results} \label{app:regul}
    Let $\Omega$ be a Lipschitz domain with a Dirichlet-Neumann partition of the boundary $\partial\Omega = \Gamma_N \cup \Gamma_D$. Both $\Gamma_D$ and $\Gamma_N$ are allowed to have vanishing measure. Let $D\in L^\infty(\Omega,\mathbb{R}^{n\times n})$ be uniformly elliptic, $k > 0$, $f_1, f_2\in L^2(I,L^2(\Omega))$ and $a_0\in L^\infty(\Omega)$, then we are interested in $L^2(I,C^0(\Omega))$, $L^\infty(I,L^\infty(\Omega))$ and $L^2(I,C^\alpha(\overline{\Omega}))$ regularity of $a_i\in H^1(I,H^1_D(\Omega),H^1_D(\Omega)')$, $i=1,2$ solving
    \begin{align}
        \int_I \langle d_t a_i, \, \cdot \, \rangle_{H^1_D(\Omega)}\,dt + \int_I\int_\Omega D\nabla a \nabla \,\cdot\, + ka\,\cdot\, dxdt &= \int_I\int_\Omega f_i\,\cdot\,dxdt \label{abstractDiffusionEquation}
        \\
        a_i(0) &= a_{0}.
    \end{align}
    We are also interested in the regularity of the difference $a_1 - a_2$. Note that $a_1 - a_2$ has better regularity properties as the initial value $a_1(0) - a_2(0) = 0$ is smooth.
    We will need varying assumptions in addition to the ones stated above, depending on the regularity we are after. Note that the main difficulty stems from the mixed boundary conditions as in this case the usual elliptic $H^2(\Omega)$ regularity fails, see for example \cite{savare1997regularity,kassmann2004regularity,grisvard2011elliptic}. Let us now state our main theorems.
    \begin{theorem}[$L^2(I,C^\alpha(\overline{\Omega}))$ Regularity]\label{L2CAlphaRegularity}
        Assume that $\Omega\subset\mathbb{R}^n$ with $n = 1, 2, 3$ is a Lipschitz domain, $\Gamma_N \cup \Gamma_D = \partial\Omega$ and assume that $\Omega\cup\Gamma_N$ is Gr\"oger regular, see \ref{GrogerReg}. Furthermore, let $f_i\in L^2(I,L^2(\Omega))$, $D\in L^\infty(\Omega,\mathcal{M}_s)$, $k > 0$ and $a_0\in L^\infty(\Omega)$. Then there is $\alpha > 0$ such that it holds $a_1 - a_2 \in L^2(I, C^\alpha(\overline{\Omega}))$ and 
        \begin{equation}\label{HolderEstimate}
            \norm{a_1 - a_2}_{L^2(I,C^\alpha(\overline{\Omega}))} \leq C \norm{f_1 - f_2}_{L^2(I,L^2(\Omega))},
        \end{equation}
        In addition, for every $\varepsilon > 0$ it holds $a_i \in L^2([\varepsilon, T], C^\alpha(\overline{\Omega}))$ and $a_i \in L^2(I, C^0(\overline{\Omega}))$.
    \end{theorem}
    \begin{theorem}\label{LinftiyIOmega} Assume that $f\in L^\infty(I\times\Omega)$, $a_0\in L^\infty(\Omega)$ and the assumptions of the beginning of the section. Then it holds that $a_i\in L^\infty(I\times \Omega)$.
    \end{theorem}
    The above regularity theorems apply to $\Tilde{a}_i$ of the main body of the article, where $\Tilde{a}_i$ is the part of the solution corresponding to homogeneous boundary conditions. Clearly all the results still hold true for $\Tilde{a}_i + 1$. To conclude we need positivity of $\Tilde{a}_i + 1$, therefore we consider a slightly different equation than \eqref{abstractDiffusionEquation}.
    \begin{theorem}[Positivity]\label{positivity}
		Assume a function $a\in H^1(I,H^1_D(\Omega),H^1_D(\Omega)')$ satisfies the following equality in the space $L^2(I,H^1_D(\Omega))'$
		\begin{align*}
			\int_I \langle d_ta,\,\cdot\, \rangle_{H^1_D(\Omega)}dt
			+
			\int_I\int_\Omega  D\nabla a\nabla\,\cdot\, 
			+
			\,k(a+1)\,\cdot\,dxdt
			=
			\int_I\int_\Omega f\,\cdot\, dxdt,
		\end{align*}
		and $a(0)+1=0$. Furthermore suppose that $f\in L^2(I,L^2(\Omega))$ is non-negative and the remaining assumptions stated at the beginning of the appendix hold true. Then $a + 1\geq 0$.
		\end{theorem}
    \subsection{Boundary Regularity} \label{app:groeger}
    We say a bounded, open set $\Omega\subset\mathbb{R}^n$ is a Lipschitz domain if $\overline{\Omega}$ is a Lipschitz manifold with boundary, see \cite[Definition 1.2.1.2]{grisvard2011elliptic}.
    In the following we will denote the cube $[-1,1]^n\subset\mathbb{R}^n$ by $Q$. The following definition is due to Gr\"oger, see \cite{groger1989aw}.
	\begin{definition}[Gr\"oger Regular Sets]\label{GrogerReg}
	    Let $\Omega\subset\mathbb{R}^n$ be bounded and open and $\Gamma\subset\partial\Omega$ a relatively open set. We call $\Omega\cup\Gamma$ Gr\"oger regular, if for every $x\in\partial\Omega$ there are open sets $U,V\subset\mathbb{R}^n$ with $x\in U$, and a bijective, bi-Lipschitz map $\phi:U\to V$, such that $\phi(x) = 0$ and $\phi(U\cap (\Omega\cup\Gamma) )$ is either $\{ x\in Q\mid x_n <0 \}$, $\{ x\in Q\mid x_n\leq 0 \}$ or $\{ x\in Q\mid x_n<0 \}\cup \{ x\in Q\mid x_n = 0, \ x_{n-1} <0 \}$.
	\end{definition}
    It can easily be seen that a Gr\"oger regular set $\Omega$ (no matter the choice $\Gamma \subset \partial\Omega$) is a Lipschitz domain, see \cite[Theorem 5.1]{haller2009holder}. The next two theorems characterize Gr\"oger regular sets in two and three dimension. We cite the results from \cite{haller2009holder}.
	\begin{theorem}[Gr\"oger Regular Sets in 2D]
	    Let $\Omega \subset \mathbb{R}^2$ be a Lipschitz domain and $\Gamma\subset\partial\Omega$ be relatively open. Then $\Omega\cup\Gamma$ is Gr\"oger regular if and only if $\overline{\Gamma}\cap(\partial\Omega\setminus\Gamma)$ is finite and no connected component of $\partial\Omega\setminus\Gamma$ consists of a single point.
	\end{theorem}
	\begin{theorem}[Gr\"oger Regular Sets in 3D]
	    Let $\Omega \subset \mathbb{R}^3$ be a Lipschitz domain and $\Gamma\subset\partial\Omega$ be relatively open. Then $\Omega\cup\Gamma$ is Gr\"oger regular if and only if the following two conditions hold
	    \begin{itemize}
	        \item [(i)] $\partial\Omega\setminus\Gamma$ is the closure of its interior.
	        \item [(ii)] For any $x\in \overline{\Gamma}\cap(\partial\Omega\setminus\Gamma)$ there is an open neighborhood $U_x$ of $x$ and a bi-Lipschitz map $\phi:U_x\cap\overline{\Gamma}\cap(\partial\Omega\setminus\Gamma)\to (-1,1)$.
	    \end{itemize}
	\end{theorem}
    \subsection{Proofs of the Regularity Results}\label{subsectionProofs}
    We will begin with the $L^2(I,C^\alpha(\overline{\Omega}))$ results, i.e., the proof of theorem \ref{L2CAlphaRegularity}. To this end we need two results from the literature, one, \cite[Theorem 3.3]{haller2009holder} on mixed elliptic boundary value problems that will yield the $C^\alpha(\overline{\Omega})$ information as well as a maximal $L^2(\Omega)$ regularity result from \cite{arendt2017jl} to transfer this to the solution of the associated parabolic equation.
    \begin{theorem}[H\"older Regularity]\label{theoremHolderReg}
        Let $\Omega\subset \mathbb{R}^n$ for $n = 2,3,4$ be open, bounded and connected and assume that $\Omega\cup\Gamma_N$ is Gr\"oger regular. Assume further, that $D\in L^\infty(\Omega,\mathbb{R}^{n\times n})$ is uniformly elliptic, $k > 0$ and let $q > n$ and denote by $q'$ its adjoint exponent, i.e., $1/q + 1/q' = 1$. Define the operator
        \begin{equation*}
            \mathcal{M}:H^1_D(\Omega) \to H^1_D(\Omega)' \quad\text{with}\quad \mathcal{M}v = \int_\Omega D\nabla v\nabla\,\cdot\, + k v\,\cdot\,dx.
        \end{equation*}
        Then there is $\alpha > 0$ such that 
        \begin{equation*}
            \mathcal{M}^{-1}: W^{1,q'}_D(\Omega)' \to C^\alpha(\overline{\Omega})
        \end{equation*}
        is well defined and continuous.
        \begin{proof}
        See \cite[Theorem 3.3]{haller2009holder}.
        \end{proof}
    \end{theorem}
    In particular, in dimensions $n = 2, 3$ we can choose $q = 4$ and by the Sobolev embedding theorems, see \cite{grisvard2011elliptic}, it holds that $L^2(\Omega)\hookrightarrow W^{1,q'}_D(\Omega)'$ and hence for $f\in L^2(\Omega)$ a solution $v\in H^1_D(\Omega)$ to 
    \[
        \int_\Omega D\nabla v\nabla\,\cdot\, + kv\,\cdot\,dx = \int_\Omega f\,\cdot\,dx\quad\text{in }H^1_D(\Omega)'
    \]
    lies in $C^\alpha(\overline{\Omega})$ and satisfies the estimate
    \[
        \norm{v}_{C^\alpha(\overline{\Omega})} \leq C \norm{f}_{L^2(\Omega)}.
    \]
    In order to amplify this elliptic regularity result to the parabolic setting, we will use an $L^2(\Omega)$ maximal regularity result. In our particular case this is a statement of the form: $f\in L^2(I,L^2(\Omega))$ implies that the solution $a\in H^1(I,H^1_D(\Omega),H^1_D(\Omega)')$ of 
    \[ 
        d_ta + \mathcal{M}a = f \quad \text{and} \quad a(0) = a_0
    \]
    is of regularity $H^1(I,L^2(\Omega))$. Consequently $f$, $d_ta$ and $\mathcal{M}a$ are members of the space $L^2(I,L^2(\Omega))$, hence the term maximal regularity. Depending on the operator $\mathcal{M}$, problems of this form can be delicate, we refer to \cite{lions2013equations} for certain results and to \cite{arendt2017jl} for a recent discussion. In our case $\mathcal{M}$ is self-adjoint and does not depend on time and therefore $L^2(\Omega)$ maximal regularity holds 
    \begin{theorem}[Maximal $L^2(\Omega)$ Regularity]\label{MaxReg} Let $\Omega \subset \mathbb{R}^n$ be a Lipschitz domain, $\partial\Omega = \Gamma_N\cup\Gamma_D$, $D\in L^\infty(\Omega,\mathcal{M}_s)$ uniformly elliptic and symmetric and $k > 0$. Define
        \begin{equation*}
            \mathcal{M}:H^1_D(\Omega) \to H^1_D(\Omega)' \quad\text{with}\quad Mv = \int_\Omega D\nabla v\nabla\,\cdot\, + k v\,\cdot\,dx.
        \end{equation*}
        and also its part in $L^2(\Omega)$ which we will denote by $M$, i.e., 
        \begin{equation*}
            \operatorname{dom}(M)\coloneqq \{ v\in H^1_D(\Omega) \mid \mathcal{M}v \in L^2(\Omega) \} \quad \text{with}\quad Mv = \mathcal{M}v
        \end{equation*}
        Then set
        \begin{equation*}
            \mathcal{X} \coloneqq H^1_0(I,L^2(\Omega))\cap L^2(I,\operatorname{dom}(M)) 
        \end{equation*}
        endowed with the norm
        \begin{equation*}
            \norm{a}_\mathcal{X}\coloneqq\norm{d_ta}_{L^2(I,L^2(\Omega))} + \norm{Ma(\cdot)}_{L^2(I,L^2(\Omega))}. 
        \end{equation*}
        Here $H^1_0(I,L^2(\Omega))$ shall denote the functions in $H^1(I,L^2(\Omega))$ with vanishing initial value. Then the space $\mathcal{X}$ is a Hilbert space and the map
        \begin{equation*}
            d_t+ M:\mathcal{X}\to L^2(I,L^2(\Omega))\quad\text{with}\quad a\mapsto d_ta + Ma
        \end{equation*}
        is a linear homeomorphism.
        \paragraph{\textbf{Addendum.}} Furthermore, if we consider the problem as above but arbitrary initial value in $H^1_D(\Omega)$, i.e.,
        \begin{equation*}
            d_ta + Ma = f \quad \text{and}\quad a(0) = a_0\in H^1_D(\Omega), 
        \end{equation*}
        the solution $a$ lies in $H^1(I,L^2(\Omega))$. 
        \begin{proof}
            This is discussed in \cite[Section 4]{arendt2017jl}. The Addendum requires $a_0$ to lie in trace space (the space of initial values) of $H^1(I,L^2(\Omega))\cap L^2(I,\operatorname{dom}(M))$. As we assume $D$ to be symmetric we have that $\mathcal{M}$ is self-adjoint and hence $H^1_D(\Omega) = \operatorname{dom}(M^{1/2})$ which coincides with the trace space, see \cite{arendt2017jl}, sections 2 and 4.
        \end{proof}
    \end{theorem}
    \begin{proof}[Proof of Theorem \ref{L2CAlphaRegularity}]
        We will first show that $a_1 - a_2 \in L^2(I,C^\alpha(\overline{\Omega}))$ and that the estimate \eqref{HolderEstimate} is satisfied. Note that $a\coloneqq a_1 - a_2$ satisfies an equation of type \eqref{abstractDiffusionEquation} with right hand side $f\coloneqq f_1 - f_2\in L^2(I,L^2(\Omega))$ and zero initial condition, namely
        \begin{equation*}
            d_ta + \mathcal{M}a = f \quad\text{in }L^2(I,H^1_D(\Omega))' \quad\text{and } a(0) = 0.
        \end{equation*}
        by Theorem \ref{MaxReg} $a$ lies in the space $\mathcal{X}$ and satisfies almost everywhere in $I$ the equation $d_ta(t) + Ma(t) = f(t)$. Using Theorem \ref{theoremHolderReg} we can estimate
        \begin{equation}\label{HolderAPrioriEsitmate}
            \norm{a(t)}_{C^\alpha(\overline{\Omega})} \leq C \norm{ f(t) - d_ta(t)}_{L^2(\Omega)}.
        \end{equation}
        It follows
        \begin{equation*}
            \norm{a}_{L^2(I,C^\alpha(\overline{\Omega}))} \leq C (\norm{f}_{L^2(I,L^2(\Omega))} + \norm{d_t}_{\mathcal{L}(\mathcal{X}, L^2(I,L^2(\Omega)))}\norm{a}_\mathcal{X} )
        \end{equation*}
        and using the continuity of $(d_t + M)^{-1}$ established in Theorem \ref{MaxReg} we get the estimate $\norm{a}_\mathcal{X}\leq C\norm{f}_{L^2(I,L^2(\Omega))}$ which yields combined with the previous estimates
        \begin{equation*}
            \norm{a_1 - a_2}_{L^2(I,C^\alpha(\overline{\Omega}))} \leq C \norm{f}_{L^2(I,L^2(\Omega))}.
        \end{equation*}
        Now we show that $a_i$ is a member of $L^2(I,C^0(\overline{\Omega}))$. We decompose $a_i$ into $a_i = a_i^0 + a_i^1$ where $a_i^0$ solves
        \begin{equation*}
            d_ta_i^0 + \mathcal{M}a_i^0 = f_i \quad \text{and} \quad a_i^0(0) = 0
        \end{equation*}
        and $a_i^1$ solves
        \begin{equation*}
            d_ta_i^1 + \mathcal{M}a_i^1 = 0 \quad \text{and} \quad a_i(0) = a_0.
        \end{equation*}
        By repeating our previous computations it is clear that $a_i^0\in L^2(I,C^\alpha(\overline{\Omega}))$. To conclude we will prove that for every $\varepsilon > 0$ it holds
        \begin{equation*}
            a_i^1 \in L^2([\varepsilon,T],C^\alpha(\overline{\Omega}))\cap L^\infty(I\times\Omega).
        \end{equation*}
        Let us begin with $a_i^1\in L^2([\varepsilon,t],C^\alpha(\overline{\Omega}))$. As $a_i^1\in C^0(I,L^2(\Omega))\cap L^2(I,H^1_D(\Omega))$ point evaluations are well defined and there is a sequence $(\varepsilon_n)$ of real numbers with $\varepsilon_n\to 0$ such that $a_i^1(\varepsilon_n)\in H^1_D(\Omega)$ for all $n\in\mathbb{N}$. For given $\varepsilon > 0$ choose $\varepsilon_n$ such that $\varepsilon_n \leq \varepsilon$. Then apply the Addendum of Theorem \ref{MaxReg} to obtain
        \begin{equation*}
            a_i^1\in H^1([\varepsilon_n,T],L^2(\Omega))\cap L^2([\varepsilon_n,T], \operatorname{dom}(M)).
        \end{equation*}
        Repeating the computations in the beginning of this proof we arrive at \eqref{HolderAPrioriEsitmate} with $a_i^1$ instead of $a$. It is then clear that $a_i^1\in L^2([\varepsilon,T],C^\alpha(\overline{\Omega}))$. The $L^\infty(I\times\Omega)$ part of the theorem holds more generally for right hand sides in $L^\infty(I\times\Omega)$ and is addressed in the following proof of Theorem \ref{LinftiyIOmega}.
    \end{proof}
    \begin{proof}[Proof of Theorem \ref{LinftiyIOmega}]
        We will use Stampacchias truncation method \cite{stampacchia1958contributi}, that is for a real number $\Bar{a}$ we will test the PDE with 
        \begin{equation*}
            (a_i - \Bar{a})^+\coloneqq \max(0, a_i - \Bar{a}) \quad\text{and} \quad (a_i + \Bar{a})^-\coloneqq -\min(0, a_i + \Bar{a}).
        \end{equation*}
        One can show that $(a_i - \Bar{a})^+$ and $(a_i + \Bar{a})^-$ are members of $H^1(I,H^1_D(\Omega),H^1_D(\Omega)')$ if $a_i$ is itself in that space and that it holds
        \begin{equation*}
            \int_0^t\langle d_ta_i, (a_i - \Bar{a})^+ \rangle_{H^1_D(\Omega)}\,dt = \frac12 \norm{ (a_i - \Bar{a})^+(t) }^2_{L^2(\Omega)} - \frac{1}{2}\norm{(a_0 - \Bar{a})^+}^2_{L^2(\Omega)}
        \end{equation*}
        and
        \begin{equation*}
            \int_\Omega D\nabla a_i(t)\nabla (a_i - \Bar{a})^+(t)\,dx = \int_\Omega D\nabla (a_i - \Bar{a})^+(t)\nabla (a_i - \Bar{a})^+(t)\,dx    
        \end{equation*}
        such as
        \begin{equation*}
            \int_\Omega a_i(t)(a_i - \Bar{a})^+(t)\,dx = \norm{(a_i - \Bar{a})^+(t)}^2_{L^2(\Omega)} + \int_\Omega \Bar{a}(a_i - \Bar{a})^+(t)\,dx
        \end{equation*}
        Hence it follows for every $t\in I$ and $\Bar{a} \geq \max(\norm{f}_{L^\infty(I\times\Omega)}, \norm{a_0}_{L^\infty(\Omega)} )$
        \begin{align*}
            \frac12 \norm{(a_i - \Bar{a})^+(t)}^2_{L^2(\Omega)} &\leq \int_0^t\int_\Omega (f - \Bar{a})(a_i - \Bar{a})^+dxds + \frac12\norm{(a_0 - \Bar{a})^+}^2_{L^2(\Omega)}
            \\&\leq 0.
        \end{align*}
        This implies that $a_i\leq \Bar{a}$ almost everywhere in $I\times\Omega$. Similarly one establishes $\Bar{a}\leq a_i$ using the test function $(a_i + \Bar{a})^-$.
    \end{proof}
    
		\begin{proof}[Proof of Theorem \ref{positivity}]
			One can check that $(a+1)^- = -\min(0, a + 1)$ is a member of $H^1(I,H^1_D(\Omega),H^1_D(\Omega)')$ and that it holds for all $t\in I$
			\begin{align*}
				\int_0^t \langle d_ta,(a+1)^- \rangle_{H^1_D}ds
				=
				\frac12\Big(\norm{(a+1)^-(0)}^2_{L^2(\Omega)} 
				-
				\norm{(a+1)^-(t)}^2_{L^2(\Omega)}
				\Big)
			\end{align*}
			and
			\begin{align*}
				&\int_0^t \int_\Omega \langle D\nabla a,\nabla (a+1)^{-} \rangle
				+
				k(a+1)(a+1)^-dxds
				\\
				=
				&-\int_0^t\int_\Omega\langle D\nabla\big[(a+1)^-\big],\nabla\big[(a+1)^-\big]\rangle
				+
				k\big[
				(a+1)^-
				\big]^2 \,dxds
				\leq 0.
			\end{align*}
			Testing the full equation with $(a+1)^-$ and using these computations one finds
			\begin{align*}
				\frac12\norm{(a+1)^-(t)}^2_{L^2(\Omega)}
				\leq
				-\int_0^t\int_\Omega f\cdot(a+1)^-\,dxds
				\leq 0.
			\end{align*}
		\end{proof}

\section{Ordinary Differential Equations}\label{sec:AppendixODE}
	The plan for this section is as follows. We recall Gr\"onwall's inequality in Theorem \ref{Gronwall} and will thereafter provide a Banach space valued ODE existence theorem in Theorem \ref{LocalExistenceTheorem}. The version of our ODE theorem guarantees existence and uniqueness of short-time solutions. We also show that the solutions depend continuously on the initial value with a common short-time existence interval, at least locally around a fixed initial value. This will come in handy for analyzing pointwise properties of certain ODEs in Lemma \ref{PointwiseProperties}. Finally we prove the existence and uniqueness of the cell ODE in Lemma \ref{SolveCellODE} and the bone ODE in \eqref{ExistenceBoneODE}
	\begin{lemma}[Gr\"onwall Variant]\label{Gronwall}
		Let $X$ be a Banach space and $I=[0,T]$ an interval. Let $x_0,y_0\in X$ and assume that $\gamma_1,\gamma_2$ are members of $L^1(I,X)$. Define $x$ and $y$ to be the integral curves
		\begin{align*}
		x(t)
		=
		x_0 + \int_0^t\gamma_1(s)\,ds
		\quad\text{and}\quad
		y(t) 
		=
		y_0 + \int_0^t\gamma_2(s)\,ds. 
		\end{align*}
		Now assume that we can estimate the integrants $\gamma_1,\gamma_2$ in the following form
		\begin{align*}
		\norm{\gamma_1(t)-\gamma_2(t)}
		\leq
		\alpha(t) + \beta(t)\norm{x(t)-y(t)}
		\quad\text{for all }t\in I,
		\end{align*}
		where $\alpha,\beta\in L^1(I)$ are non-negative functions. Then it holds that
		\begin{align*}
		\norm{x(t)-y(t)}\leq C\big(\norm{x_0-y_0}+\norm{\alpha}_{L^1(I)}\big)
		\quad\text{for all }t\in I
		\end{align*}
		and the constant $C$ can be chosen to be $C=1+\norm{\beta}_{L^1(I)}\exp(\norm{\beta}_{L^1(I)})$.
		\begin{proof}
			Just write out the estimate that the difference $\norm{x(t)-y(t)}$ satisfies due to the assumptions and then use the usual integral formulation of Gr\"onwall's inequality, see for example \cite[Theorem 1.2.8]{qin2017analytic}.
		\end{proof}
	\end{lemma}
	\begin{theorem}[Local Existence]\label{LocalExistenceTheorem}
	Let $X$ be a Banach space, $I=[a,b]$ a bounded interval and $F:I\times X\to X$ a Carath\'eodory function, i.e., $F(\cdot, x)$ is Bochner measurable for all $x\in X$ and $F(t,\cdot)$ is continuous almost everywhere in $I$. Assume that for every bounded set $B\subset X$ there are functions $m_B\in L^p(I)$, $p\in[1,\infty]$ and $L_B\in L^1(I)$, possibly depending on $B$, such that
	\begin{align}
	    \norm{F(t,x)}_X &\leq m_B(t) \quad \text{a.e.\ in }I,\ \forall x\in B \label{localExistenceCondition1},
	    \\
	    \norm{F(t,x) - F(t,y)}_X &\leq L_B(t)\norm{x-y} \quad \text{a.e.\ in }I,\ \forall x,y\in B\label{localExistenceCondition2}.
	\end{align}
	Let furthermore $t_0\in I$, $x_0\in X$ and $R > 0$ be arbitrary, then there exists a time interval $I_\delta\coloneqq[t_0-\delta,t_0+\delta]\cap I$ such that for any initial value $y_0\in \overline{B_R(x_0)}\subset X$ there is a unique short time solution $x\in W^{1,p}(I_\delta,X)$ of the ODE
	\begin{align*}
	    d_tx(t) = F(t,x(t)) \quad\text{and } x(t_0) = y_0.
	\end{align*}
	Moreover the map $y_0\mapsto x(y_0)$ taking the initial value to its solution seen as a map $\overline{B_R(x_0)}\subset X \to C^0(I_\delta,X)$ is continuous.
	\end{theorem}
\begin{remark}
    Note that the Lipschitz assumption \eqref{localExistenceCondition2} implies that $x\mapsto F(t,x)$ is continuous. Therefore, to establish the Carath\'eodory regularity of $F$ it is enough to provide the Bochner measurability of the maps $F(\cdot,x):I\to X$ for all $x\in X.$
\end{remark}
	\begin{proof}
	First we clarify the dependence of $\delta$; we choose it to satisfy
	\begin{align}\label{SizeOfDelta}
	    \int_{I_\delta} m_{B_{2R}(x_0)}(s)\,ds \leq R \quad\text{and}\quad  \int_{I_\delta}L_{B_{2R}(x_0)}(s)\,ds < 1.
	\end{align}
	Then we consider the complete metric space $E\subset C^0(I_\delta,X)$ given by
	$$
	E\coloneqq \{ x\in C^0(I_\delta,X) \mid \sup_{t\in I_\delta} \norm{x(t)-x_0} \leq 2R \}
	$$
	and define the map
	$$
	\Phi: E \to E \quad \text{with} \quad \Phi(x)(t) = y_0 + \int_{t_0}^t F(s,x(s))\,ds.
	$$
	We now proceed by showing the following facts
	\begin{enumerate}
	    \item [(i)] For all $x\in E$ the map $t\mapsto F(t,x(t))$ is Bochner integrable as a map $I_\delta \to X$. In fact it is a member of $L^p(I_\delta, X)$.
	    \item[(ii)] The function $\Phi$ is a self-mapping and a contraction.
	    \item[(iii)] The fix-point of $\Phi$ is a member of $W^{1,p}(I_\delta, X)$ and corresponds to the solution of the ODE.
	    \item[(iv)] The solution depends continuously on the initial data.
	\end{enumerate}
	To establish (i) note that the assumption of Carath\'eodory regularity of $F$ implies that $t\mapsto F(t,x(t))$ is Bochner measurable as a map $I_\delta \to X$ for all maps $x:I_\delta \to X$ that are itself Bochner measurable, which clearly holds for members of $E$. We are left to show the integrability, so we estimate for $x\in E$
	$$
	\int_{I_\delta} \norm{F(t,x(t))}^p_X\,dt \leq \int_{I_\delta}|m_{B_{2R}(x_0)}(t)|^p\,dt < \infty,
	$$
	which shows the assertion. Now to (ii). Let again $x\in E$ and estimate
	\begin{align*}
	\sup_{t\in I_\delta} \norm{\Phi(x)(t) - x_0}_X &\leq \norm{x_0-y_0} + \int_{I_\delta} \norm{ F(t,x(t)) }_X\,dt 
	\\ 
	&\leq R + \int_{I_\delta}m_{B_{2R}(x_0)}(t)\,dt 
	\\
	&\leq 2R.
	\end{align*}
	To see that $\Phi$ is a contraction compute for $x,y\in E$
	\begin{align*}
	\sup_{t\in I_\delta} \norm{\Phi(x)(t) - \Phi(y)(t)}_X &\leq \sup_{t\in I_\delta} \int_{I_\delta}\norm{F(t,x(t)) - F(t,y(t))}_X\,dt
	\\&\leq 
	\underbrace{\norm{ L_{ B_{2R}(x_0) } }_{L^1(I_\delta)}}_{<1}\norm{x - y}_E
	\end{align*}
	The claim (iii) follows as the unique fix-point $x$ of $\Phi$ is, by the fundamental theorem, a solution to the ODE. As $d_tx(t) = F(t,x(t))$ the $L^p$ integrability of the derivative of this fix-point follows from the one of $F(\cdot,x(\cdot))$ which was established in (i).
	\\
	Finally to (iv), where we will employ Gr\"onwall's lemma. Let $y_0$ and $\overline{y}_0$ be in $\overline{B_R(x_0)}$, then the according solutions are given by
	$$
	y(t) = y_0 + \int_{t_0}^t F(s,y(s))\,ds \quad \text{and} \quad \overline{y}(t) = \overline{y}_0 + \int_{t_0}^t F(s,\overline{y}(s))\,ds.
	$$
	The difference in the integrands can be estimated by 
	\begin{align*}
	    \norm{F(t,y(t)) - F(t,\overline{y}(t))}_X \leq C L_{ B_{2R}(x_0) }(t) \norm{y(t) - \overline{y}(t)}_X.
	\end{align*}
	So applying Lemma \ref{Gronwall} with $\alpha = 0$ and $\beta = L_{ B_{2R}(x_0) }$ yields
	$$
	\norm{y(t) - \overline{y}(t)}_X \leq C \norm{y_0 - \overline{y}_0}_X.
	$$
	\end{proof}
\begin{remark}It is often of interest to show the existence of long-time solutions. A particularly simple case in the setting of the above theorem is encountered if it holds that for any initial value $y_0\in \overline{B_R(x_0)}\subset X$ the solution takes values only in $\overline{B_R(x_0)}$. Then one glues together multiple short-time solutions with the guarantee of $\delta$ not deteriorating.  
\end{remark}
	\begin{lemma}[Pointwise Properties of Solutions]\label{PointwiseProperties}
	    Let $I=[a,b]$ be an interval and $K:I\times\mathbb{R}\to\mathbb{R}$ a Carath\'eodory function such that $x\geq 0$ implies $K(t,x) \geq 0$ for all $t\in I$. For fixed $t_0\in I$ consider the ODE
	    $$
	    x'(t) = K(t,x(t))\bigg( 1- \frac{x(t)}{\theta} \bigg) \quad\text{and}\quad x(t_0) = x_0,
	    $$
	    where $\theta >0$ and $\lambda\geq0$ are fixed numbers and $x_0\in[0,\theta]$. Assume that there is an interval $I_\delta = [t_0-\delta,t_0+\delta]\cap I$ such that for any choice of $x_0\in[0,\theta]$ we have a solution $x\in W^{1,1}(I_\delta)$ of the ODE which we assume to continuously depend on the initial data $x_0$, i.e., we assume that for every $x_0\in [0,\theta]$ there is a neighborhood $N_{x_0}$ around $x_0$ such that $x_0\mapsto x$ is continuous as a map $N_{x_0} \to C^0(I_\delta)$, where $x$ is the solution to the ODE with initial value $x_0$. Then it holds
	    $$
	    0\leq x(t) \leq \theta \quad \text{for all }t\in I_\delta.
	    $$
	    \begin{proof}
	    We know that $x$ satisfies the identity
	    $$
	    x(t) = x_0 + \int_{t_0}^t K(s,x(s))\bigg(1-\frac{x(s)}{\theta}\bigg)\,ds \quad \text{for all }t\in I_\delta
	    $$
	    \emph{Upper Barrier.} We prove this by contradiction. Suppose there was $s_0\in I_\delta$ with $x(s_0) > \theta$, then on a neighborhood of $s_0$ solution must be non-increasing which can be seen as follows: Due to the continuity of $x$ there is $\varepsilon>0$ such that
	    $$
	    x(t) \geq \theta \quad\text{for all }t\in (s_0-\varepsilon,s_0 + \varepsilon).
	    $$
	    If $x$ was not non-increasing on $(s_0-\varepsilon,s_0+\varepsilon)$ then there exist $t_1 < t_2$ in that interval such that $x(t_2) > x(t_1)$ and therefore 
	    $$
	    0 < x(t_2) -x(t_1) = \int_{t_1}^{t_2} \underbrace{K(s,x(s)) }_{\geq 0} \underbrace{ \bigg(1 - \frac{x(s)}{\theta} \bigg) }_{\leq 0}\,ds \leq 0,
	    $$
	    which settles the claim. Now, by judicious Zornification we produce a maximal interval $Z$ around $s_0$ on which $x$ is non-increasing. Then $t^*\coloneqq \inf Z = t_0$ (if it was not $t_0$, repeat the above reasoning and find that $Z$ was not maximal) and hence $\theta < x(t^*) = x(t_0) \leq \theta$ clearly is a contradiction. 
	    \\
	    \\
	    \emph{Lower Barrier.} With an analogue reasoning as in the proof for the upper barrier we can establish the following: If $x(s_0)\in (0,\theta]$ for some $s_0\in I_\delta$ then $x(t)\in[x(s_0),\theta]$ for all $t\geq s_0$. This yields the claim for all initial values strictly larger than zero. We need $x(s_0)$ to exceed zero to guarantee the existence of a small interval $(s_0-\delta, s_0+\delta)$ where $x\geq 0$ still holds, to be able to use $K(s,x(s))\geq 0$ on this interval. For $x(t_0) = 0$ we approximate the solution by considering initial values $x_n(t_0) = 1/n$, i.e., we find solutions $x_n$ to 
	    $$
	    x'_n(t) = K(s,x_n(s))\bigg( 1 - \frac{x_n(s)}{\theta}\bigg) \quad \text{with} \quad x_n(t_0) = n^{-1}.
	    $$
	    As shown before we then know that $x_n(t)\in[1/n,\theta]$ for all $t\in I_\delta$. By the continuity we assumed we can pass to the limit in $n$ and obtain $0\leq x(t) \leq \theta$ for all $t\in I_\delta$.
	    \end{proof}
	\end{lemma}
		We continue by explaining the connection between Banach space valued ODEs and the formulation as a family of real valued ODEs in our examples \eqref{CellODE}, \eqref{ODE}. A moments reflection reveals that we only need to guarantee that for every fixed $x\in\overline{\Omega}$ it holds
	\[
	    t\mapsto d_tb(t,x) = \frac{d}{dt}(t\mapsto b(t,x)),
	\]
	where on the right hand side we talk about the usual, real valued, weak derivative. Therefore let $b\in W^{1,2}(I,C^0(\overline{\Omega}))$, then for $d_tb$ it holds by definition 
	\[
		\int_I d_tb(t)\varphi(t)\,dt = -\int_I b(t)\partial_t\varphi(t)\,dt \quad \forall \varphi \in\mathcal{D}(I).
	\]
	The integral used above is the $C^0(\overline{\Omega})$ valued Bochner integral, thus using that point evaluation is a linear and continuous map on the space of continuous functions we find that for every $x\in \overline{\Omega}$ it holds
	\begin{align*}
		\int_I d_tb(t)(x)\varphi(t)\,dt
		=
		-\int_I b(t)(x)\partial_t\varphi(t)\,dt
		\quad\forall \varphi\in\mathcal{D}(I),
	\end{align*}
	meaning that for every fixed $x\in\overline{\Omega}$ the function $t\mapsto b(t)(x)$ satisfies the real valued ODE as desired. If we additionally assume that both, the Banach space valued ODE and the scalar one have unique solutions, we obtain that both settings are equivalent. This means that the above Lemma is applicable to deduce pointwise properties of the solutions to the Banach space valued ODEs.
	\begin{lemma}[Solveability of the Cell ODE]\label{SolveCellODE}
	Assume $H: \mathbb{R}^{N + 2} \to \mathbb{R}$ is locally Lipschitz continuous and that $H$ is non-negative if all its arguments are non-negative. Further, let the assumptions \ref{AssumptionOnKAndH} be satisfied. Then there is a unique solution $c \in W^{1,p}(I, C^0(\overline{\Omega}))$ of the equation
	\begin{equation*}
	    d_tc = H(a_1,\dots,a_N,b,c)\bigg(1 - \frac{c}{1-\rho}\bigg) \quad \text{with} \quad c(0) = 0.
	\end{equation*}
	Furthermore the solution satisfies $0\leq c(t,x) \leq 1 - \rho(x)$ for all $t\in I$ and $x\in\overline{\Omega}$.
	\begin{proof}
	 To begin with, define the auxiliary function
	 \begin{equation*}
	     \tilde{H}:\mathbb{R}^{N+2}\times[c_P,C_P]\to\mathbb{R}\quad\text{with}\quad \tilde{H}(a,b,c,\rho) = H(a,b,c)\bigg( 1 - \frac{c}{1-\rho} \bigg).
	 \end{equation*}
	 Then $\tilde{H}$ is locally Lipschitz continuous. Now note that the ODE is induced by
	 \begin{equation*}
	     F:I\times C^0(\overline{\Omega})\to C^0(\overline{\Omega}) \quad\text{with}\quad F(t,c) = x \mapsto \tilde{H}(a(t,x),b(x),c(x),\rho(x)).
	 \end{equation*}
	 We aim to apply Theorem \ref{LocalExistenceTheorem} to produce a short-time solution, hence we need to guarantee
	 \begin{enumerate}
	     \item [(i)] $F(t,c)\in C^0(\overline{\Omega})$ for all $t\in I$ and $c\in C^0(\overline{\Omega})$,
	     \item [(ii)] $F(\,\cdot\,,c):I\to C^0(\overline{\Omega})$ is Bochner measurable for all $c\in C^0(\overline{\Omega})$,
	     \item[(iii)] $F$ satisfies \eqref{localExistenceCondition1} and \eqref{localExistenceCondition2}.
	 \end{enumerate}
	 The statement (i) is clear as $F(t,c)$ is a composition of continuous functions. To prove (ii) we write $a$ as a pointwise almost everywhere limit of finitely valued, measurable functions $(s_k)\subset \mathcal{S}(I, C^0(\overline{\Omega})^N)$. Then $t\mapsto \tilde{H}(s_k(t),b,c,\rho)$ is still a member of $\mathcal{S}(I, C^0(\overline{\Omega})^N)$. As $s_k(t)\to a(t)$ in $C^0(\overline{\Omega})$ almost everywhere in $I$, for fixed $t\in I$ the set
	   $$
	   \bigcup_{k\in\mathbb{N}}\{ (s_k(t,x),b(x),c(x),\rho(x)),(a(t,x),b(x),c(x),\rho(x)) \mid x\in\overline{\Omega} \} \subset \mathbb{R}^{N+3}
	   $$
	   is relatively compact in $ \mathbb{R}^{N+3} $. Hence it holds
	   \begin{align*}
	       \norm{\tilde{H}(s_k(t),b,c,\rho) - \tilde{H}(a(t),b,c,\rho)}_{C^0(\overline{\Omega})}
	       \leq
	            C\norm{s_k(t)-a(t)}_{C^0(\overline{\Omega})}.
	   \end{align*}
	   This establishes the Bochner measurability in $(ii)$. To show (iii) let $B\subset C^0(\overline{\Omega})$ be bounded and compute 
	   \begin{align*}
	       \norm{F(t,\tilde{c})}_{C^0(\overline{\Omega})} &\leq \norm{ H( a(t), b(t), \tilde{c} ) }_{C^0(\overline{\Omega})} \norm{1 - \frac{ \tilde{c} }{ 1 - \rho }}_{C^0(\overline{\Omega})}
	       \\ &\leq m_B^H(t)\big( 1 + \max_{x \in \overline{\Omega}}[1/(1-\rho(x))]\norm{\tilde{c}}_{C^0(\overline{\Omega})} \big)
	       \\ &\leq C m_B^H(t).
	   \end{align*}
    In a similar way we estimate
    \begin{align*}
        \norm{ F(t,\tilde{c}) - F(t, \tilde{\tilde{c}}) }_{C^0(\overline{\Omega})} &\leq \norm{ H( a(t),b(t),\tilde{c} ) - H( a(t),b(t),\tilde{\tilde{c}})}_{C^0(\overline{\Omega})} \norm{ 1 - \frac{\tilde{c}}{1 - \rho} }
        \\&+
        \norm{H( a(t), b(t), \tilde{c} )}_{C^0(\overline{\Omega})}\norm{\frac{\tilde{c}}{1 - \rho} - \frac{\tilde{\tilde{c}}}{1 - \rho} }_{C^0(\overline{\Omega})}
        \\&\leq
        CL_B^H(t)\norm{\tilde{c} - \tilde{\tilde{c}}}_{{C^0(\overline{\Omega})}} + Cm_B^H(t)\norm{\tilde{c} - \tilde{\tilde{c}}}_{{C^0(\overline{\Omega})}}
        \\&\leq
        \underbrace{ C\max(L_B^H(t), m_B^H(t))}_{= \zeta}\norm{\tilde{c} - \tilde{\tilde{c}}}_{{C^0(\overline{\Omega})}}
    \end{align*}
    and $\zeta$ is a member of  $L^1(I)$. To establish a long-time solution note that by our pointwise lemma \ref{PointwiseProperties} we have
	 \begin{equation*}
	 0 \leq c(t,x) \leq 1 - \rho(x) \leq 1.
	 \end{equation*}
	 Remember that we discussed the connection between Banach space valued ODEs and $\mathbb{R}$ valued ODEs in the section \ref{sectionMathematicalFormulation}. Finally using the remark following Theorem \ref{LocalExistenceTheorem} we conclude.
	\end{proof}
	\end{lemma}
	Clearly the bone ODE can now be treated identically, provided one assumes the same for $K$ as one did for $H$. 
	\begin{lemma}[Solveability of the Bone ODE]\label{ExistenceBoneODE}
	Assume $K: \mathbb{R}^{N + 2} \to \mathbb{R}$ is locally Lipschitz continuous and that $K$ is non-negative if all its arguments are non-negative. Further, let the assumptions \ref{AssumptionOnKAndH} be satisfied. Then there is a unique solution $b \in W^{1,q}(I, C^0(\overline{\Omega}))$ of the equation
	\begin{equation*}
	    d_tb = H(a_1,\dots,a_N,b,c)\bigg(1 - \frac{b}{1-\rho}\bigg) \quad \text{with} \quad b(0) = 0.
	\end{equation*}
	Furthermore the solution satisfies $0\leq b(t,x) \leq 1 - \rho(x)$ for all $t\in I$ and $x\in\overline{\Omega}$.
	\end{lemma}

\bibliographystyle{abbrv}
\bibliography{bib}

\end{document}